\documentclass[12pt]{amsart}

\usepackage{caption,amsmath,amssymb,amsfonts,amsthm,latexsym,graphicx,multirow,color,enumerate}

\usepackage[colorlinks, citecolor=blue,linkcolor=blue, anchorcolor=blue]{hyperref}

\usepackage[all]{xy}
\usepackage{mathrsfs}
\usepackage{float}

\usepackage{longtable}
\usepackage{xcolor}
\usepackage{graphics}
\usepackage{ulem}
\usepackage{booktabs}

\usepackage{titlesec}
\titleformat{\appendix}
{\normalfont\bfseries}{\theappendix}{1em}{}

\titleformat{\section}
  {\normalfont\bfseries}{\thesection}{1em}{}
\titleformat{\subsection}
  {\normalfont\bfseries}{\thesubsection}{1em}{}
  \titleformat{\subsubsection}
  {\normalfont\bfseries}{\thesubsubsection}{1em}{}

\definecolor{newblue}{RGB}{30,144,255}

\definecolor{blue3}{RGB}{0,10,100}

\definecolor{blue1}{RGB}{100,10,250}

\def\ov{\overline} 
\def\l{\langle} \def\r{\rangle} 
 
\def\FF{\mathbb F}
\def\ZZ{{\bf Z}}

\def\mod{{\sf mod~}}

\def\Aut{{\sf Aut}} 
\def\Out{{\sf Out}}
\def\Cos{{\sf Cos}}

 \def\soc{{\sf soc}}

\def\D{{\rm D}} 
\def\S{{\rm S}}
 \def\M{{\rm M}}

\def\C{{\bf C}}\def\N{{\bf N}}\def\Z{{\bf Z}}

\def\Ome{{\it \Omega}}
\def\Ga{{\it \Gamma}}

\def\a{\alpha} \def\b{\beta} \def\g{\gamma} \def\s{\sigma}
  \def\o{\omega}

 \def\GL{{\rm GL}}

\def\Sp{{\rm Sp}}\def\PSp{{\rm PSp}}

\def\PGSp{{\rm PGSp}}

\def\PGammaL{{\rm P\Gamma L}}

\def\A{{\rm A}}
\def\Sym{{\rm Sym}}

\def\PSL{{\rm PSL}}  \def\PGL{{\rm PGL}}
\def\GL{{\rm GL}} \def\SL{{\rm SL}}
\def\AGL{{\rm AGL}}
  
\def\GU{{\rm GU}}   \def\PSO{{\rm PSO}} \def\SO{{\rm SO}}

  \def\D{{\rm D}}

\def\rmC{{\rm C}}

\def\calM{{\mathcal M}}

\def\M{{\sf M}}

\def\soc{{\rm soc}}

\def\le{\leqslant}
\def\ge{\geqslant}

\def\RotaMap{{\sf RotaMap}}
\def\BiRoMap{{\sf BiRoMap}}

\def\core{{\sf core}}

\def\cal M{{\mathcal M}}
\def\a{\alpha}
\def\ov{\overline}

\newcommand{\rot}{\textup{rot}}

\newcommand{\inv}{{\sf Inv}}

\def\rot{{\sf rot}}

\def\chi{{\sf chi}}
\def\Reg{{\sf reg}}

\def\leq{\leqslant}
\def\geq{\geqslant}

\def\cayM{{\sf CayM}}
\def\CayM{{\sf CayM}}

\newtheorem{theorem}{Theorem}[section]%
\newtheorem{lemma}[theorem]{Lemma}%
\newtheorem{corollary}[theorem]{Corollary}%
\newtheorem{proposition}[theorem]{Proposition}%
\newtheorem{definition}[theorem]{Definition}%
\newtheorem{problem}[theorem]{Problem}%
\newtheorem{example}[theorem]{Example}%
\newtheorem{construction}[theorem]{Construction}%
\newtheorem{remark}[theorem]{Remark}

\def\qed{{\hfill$\Box$\smallskip}
	\medbreak}

\begin{document}
	
	\title{Regular, and (bi-)rotary Hall Cayley maps}
	\thanks{This work was supported by NSFC grants 12461061, and 11931005.}
	
	\author{Wendi Di}
	\address{ 
        School of Mathematics\\
        East China  University of Science and Technology, Shanghai 200237, P. R. China
        }
	\email{diwd@ecust.edu.cn}
	
	\author{Zheng Guo}
	\address{School of Mathematics and Statistics \\
		Zhengzhou University,  Zhengzhou 450001,  P.R.China}
	\email{12131227@mail.sustech.edu.cn}
	
	\author{Cai Heng Li}
	\address{Department of Mathematics \\
		SUSTech International Center for Mathematics\\
		Southern University of Science and Technology,  Shenzhen 518055,  P.R.China}
	\email{lich@sustech.edu.cn}

\begin{abstract}

A classification is given of regular, rotary, and birotary Cayley maps of which the vertex number is coprime to the valency.
The classification is involved in constructions of new examples of interesting Cayley maps.
		
	\end{abstract}
	
	\maketitle
	\date\today

\section{Introduction}

Let $\calM=(V,E,F)$ be a map with vertex set $V$, edge set $E$ and face set $F$, so that $\Ga=(V,E)$ is the underlying graph of $\calM$.
Assume that $|V|\ge2$ and $|F|\ge2$.
For a vertex $\a\in V$, let $\Ga(\a)$ be the set of vertices adjacent to $\a$, and let $E(\a)$ be the set of edges which are incident with $\a$.

Let $\Aut\calM$ be the automorphism group of $\calM$.
An automorphism of $\calM$ stabilizing a vertex $\a$ is a rotation or a reflection of a disc surrounding $\a$, and hence the stabilizer $(\Aut\calM)_\a$ is a cyclic group or a dihedral group.
If $G\leq\Aut\calM$ is transitive on the edge set or the arc set, then $\calM$ is said to be {\it $G$-edge-transitive} or {\it $G$-arc-transitive}, respectively.
An arc-transitive map is edge-transitive, but the converse statement is not necessarily true.

Edge-transitive maps and arc-transitive maps have received considerable attention in the literature, refer to \cite{CFLZ,Edge-t-maps} for the former, and \cite{Reg-maps,Siran2025,JLSW19,LP25} for the latter.
In particular, edge-transitive maps are categorized into fourteen types in \cite{Edge-t-maps}, among which five types are arc-transitive.
In this paper we focus on three types of arc-transitive maps, namely, regular maps, rotary maps and bi-rotary maps.
We remark that rotary maps are also called orientably regular.
In terms of \cite{Edge-t-maps}, regular maps are of type~1, rotary maps are of type $2^*$ex, and bi-rotary maps are of type $2^P$ex.

Regular maps and rotary maps have been extensively studied  in the literature, refer to \cite{Reg-maps,LPS} for generic constructions and characterizations, and to \cite{conder2009,Siran2025,JLSW19,james1985,kovacs2017,Li-2006,Li-2008,LP25,SS24} for the study on various families of regular maps and rotary maps.
For the study of bi-rotary maps, refer to \cite{CDL,Birotary,LPS}.
Another type of edge-transitive maps which has taken considerable attention is regular dessins, see \cite{CFLZ,Jones-dessin} and the references therein.



Let $\calM$ be a map.
An edge $e$ with end vertices $\a,\b$ is written as $e=[\a,e,\b]$, and corresponds to two arcs $(\a,e,\b)$ and $(\b,e,\a)$.
Let $f,f'$ be the two faces which are incident with the edge $e$.
Then the triples $(\a,e,f)$ and $(\a,e,f')$ are the two {\it flags}  incident with the arc $(\a,e)$.

Assume that a subgroup $G\le\Aut\calM$ is arc-transitive on $\calM$, and $G_\a$ is cyclic for $\a\in V$, so that $\calM$ is a {\it $G$-vertex-rotary} map.
Then for an edge $e=[\a,e,\b]$, the vertex stabilizer $G_\a=\l a\r$ is cyclic, and the edge stabilizer $G_e=\l z\r\cong\ZZ_2$ such that $z$ interchanges the paired arcs $(\a,e,\b)$ and $(\b,e,\a)$.
In this case, $G=\l a,z\r$, namely, $G$ is generated by two elements one of which is an involution, leading to the following definition: 

\begin{definition}\label{def:rotary-pairs}
{\rm
For a finite group $G$, if $G=\l a,z\r$ with $|z|=2$ then $(a,z)$ is called a {\it rotary pair} for $G$.
Two rotary pairs $(a_1,z_1)$ and $(a_2,z_2)$ for $G$ are said to be {\it equivalent} if $(a_1,z_1)^\s=(a_2,z_2)$ for some $\s\in\Aut(G)$.
Then a rotary pair $(a,z)$ for $G$ is called {\it reflexible} if $(a,z)$ and $(a^{-1},z)$ are equivalent.
}
\end{definition}

By \cite[Theorem\,1.1]{LPS}, each rotary pair $(a,z)$ for a group $G$ defines a {\it coset graph} $\Ga=\Cos(G,\l a\r, \l z \r)$, which has vertex set $V=[G:\l a\r]$ and edge set $E=[G:\l z\r]$ such that
$\l a\r x$ is incident with $\l z\r y$  if and only if $\l a\r x\cap \l z\r y\neq \emptyset$.
Label the vertex corresponding to $\l a\r$ as $\a$.
Then $G_\a=\l a\r$ acts regularly on $E(\a)$, and $G$ is arc-regular on $\Ga$.
Further, by \cite[Theorem\,1.9]{LPS}, for each $G$-vertex-rotary map $\calM$ there exists a rotary pair $(a,z)$ for $G$ such that the underlying graph is $\Cos(G,\l a\r,\l z\r)$, and either 
\begin{enumerate}[(A)]
\item $\calM$ has face set $[G:\l az\r]$ and is a rotary map $\RotaMap(G,a,z)$, or
\item $\calM$ has face set $[G:\l z,z^a\r]$ and is a bi-rotary map $\BiRoMap(G,a,z)$.
\end{enumerate}
This shows that studying rotary maps and bi-rotary maps is equivalent to studying groups with rotary pairs.
A $G$-vertex-rotary map $\cal M$ defined by a rotary pair $(a,z)$ is a {\it regular map} if and only if $(a,z)$ is reflexible.


A map $\calM=(V,E,F)$ is called a {\it Cayley map} of a group $H$ if $\Aut\calM$ contains a subgroup $H$ that is regular on $V$.
We remark that, by definition, a Cayley map may have multi-edges, and so its underlying graph $\Ga=(V,E)$ is a {\it Cayley (multi-)graph}, which is simply called a {\it Cayley graph} (even if it has multi-edges).
Cayley maps were introduced by Biggs in \cite{Cayley-map-biggs}, and have been extensively studied since then, see for instance \cite{BCMG2022,conder2009,conder2007,DuYuLuo23,Cayley-map-jajcay,Jajcay2003,kovacs2017,Li-2006,Cay-maps,Balanced-Cay-2,Jozef92,stephens2014}, and the references therein.


In this paper, we study Cayley maps which are rotary or bi-rotary. 


\begin{definition}\label{def:1}
{\rm
A Cayley map $\calM$ of a group $H$ is called a {\it vertex-rotary Hall Cayley map} if there exists a subgroup $G\le \Aut\calM$ which is $G$-vertex-rotary on $\calM$ and contains a Hall subgroup isomorphic to $H$.
In this case, $G=\l a,z\r=H\l a\r$ with $H\cap \l a\r=1$, and $\calM$ is denoted by $\calM=\CayM(G,H,a,z)$.
For convenience, $G=H\l a\r$ is called a {\it Hall factorization}, $(G,H,a)$ a {\it Hall triple}, $(a,z)$ a {\it Hall rotary pair}, and $\l a\r$ a {\it Hall cycle}.
In addition, a Hall rotary pair $(a,z)$ is a {\it Hall reflexible pair} if $(a,z)$ is reflexible.
}
\end{definition}


In order to state our main theorems, we introduce the following notation.
\begin{definition}\label{def:1.2}
{\rm
Let $G$ be a finite group, and let 
\begin{align*}
    &\rot(G) \text{ denote the number of inequivalent Hall rotary pairs $(a,z)$ for $G$;}\\
    &\Reg(G) \text{ denote the number of inequivalent Hall reflexible pairs $(a,z)$ for $G$;}\\
    &{\inv}(S) \text{ denote the number of involutions in a subset $S\subset G$.}
\end{align*}
}
\end{definition}
For a positive integer $r$, let
\[\delta(r,i)=
\left\{
\begin{array}{ll}
1 & \mbox{if $r \equiv i \pmod{4}$, or}\\
0 & \mbox{otherwise.}
\end{array}
\right.
\]

The following two theorems  count the inequivalent Hall  rotary pairs and Hall reflexible pairs of almost simple groups and general finite groups, respectively.

\begin{theorem}\label{thm:rotG}
Suppose  $G$ is an almost simple group which has Hall rotary pairs.  
Then one of the following holds:
\begin{enumerate}[\rm(i)]
\item $(G,\rot(G),\Reg(G))$ is as follows;
\begin{table}[h!]
\centering

\renewcommand{\arraystretch}{1.1}
\begin{tabular}{|c|c c c c c c c| }
\hline
$G$&$\A_5$, $\S_5$ & $\PSL(2,11)$ & $\M_{11}$ & $\M_{23}$ & $\A_{11}$ & $\S_{11}$ & $\A_{23}$  \\
\hline
$\rot(G)$& $2$ & $5$ & $20$ & $330$ & $1635$ & $1579$ & $72633517935$ \\
\hline
$\Reg(G)$ & $2$ & $5$ & $0$ &$0$ & $155$ & $155$ & $2611103$\\
\hline
\end{tabular}
\end{table}

\item $(G,\rot(G))$ lies in Table~$\ref{tab:rot}$, and $(G,\Reg(G))$ lies in Table~$\ref{tab:reg}$, where the number of involutions $\inv(\cdot)$ is given in Lemmas~$\ref{prop:PSL}$, $\ref{lem:Ap}$,  and $\ref{lem:invA}$.
\end{enumerate}
\end{theorem}

{\small
\begin{table}[h!]
\centering
\caption{\small The values of $\rot(G)$}
\label{tab:rot}
\resizebox{\textwidth}{!}{
\begin{tabular}{|c|c|c|}
\hline
   $G$ & $\rot(G)$  &Conditions \\
\hline

 $\PSL(2,q)$ &
${1\over 2f}(q-2)\cdot \phi(q+1)$
&$q = 2^f$ \\
\hline
 $\PSL(2,q){:}2$ &
${1\over2f}q^{1/2}\cdot\phi(q+1)$&$q = 2^{f}$, $f$ even \\

\hline
 $\PSL(d,q)$ &
${q-1\over 2df(q^d-1)}\inv(\PSL(d,q))\cdot\phi({q^d-1\over q-1})$&$d>2$ prime \\

\hline

 \multirow{2}{*}{$\PSL(d,q){:}2$} &
\multirow{2}{*}{${q-1\over 2df(q^d-1)}\Big(\inv(\PSL(d,q){:}2)-\inv(\PSL(d,q))-\frac{q^{d/2}+1}{q^{1/2}+1}\Big)\cdot\phi({q^d-1\over q-1})$}&
$d>2,\ q=p^{f}$, \\
&& $f$ even\\
\hline
 $\A_r$ &
${1\over r}\inv(\A_r) -  \delta(r,1)$
&$\begin{aligned} &\quad\, \,\,r\neq 11,23,\\ &\text{not Singer prime} \end{aligned}$ \\
\hline

 $\A_r$ & ${1\over r}\inv(\A_r)-\frac{r-1}{r}\sum_{{q^d-1\over q-1}=r}\frac{\inv(\PSL(d,q))}{d^{t+1}}-\delta(r,1)$  
 &$\begin{aligned} &\ r\neq 2^{2^t}+1,\\ &\text{Singer prime} \end{aligned}$ \\
\hline
$\A_r$&${1\over r}\inv(\A_r)-{2^{2^t-t-1}\over r}\inv(\PSL(2,q){:}2)-(2^{2^t-t-1}-1)$&  $\begin{aligned} &\quad r\neq5,\\ &r=2^{2^t}+1 \end{aligned}$ \\
\hline

 $\S_r$ &
${1\over r}\inv(\S_r)-{1\over r}\inv(\A_r)-\delta(r,3)$
&  \\

\hline
\end{tabular}}
\end{table}

\begin{table}[htbp]
\centering
\caption{The values of $\Reg(G)$}
\label{tab:reg}
\resizebox{\textwidth}{!}{
\begin{tabular}{|c|c|c|}
\hline
 $G$ &  $\Reg(G)$ &Conditions \\
\hline

$\PSL(2,q)$ & ${1\over 2f}(q-2)\cdot \phi(q+1)$ &$q = 2^f$ \\
\hline

$\PSL(2,q){:}2$ & ${1\over2f}\cdot q^{1/2}\cdot \phi(q+1)$ & $q = 2^{f}$, $f$ even \\
\hline

$\PSL(d,q)$ & $ \frac{1}{2df}\cdot\inv(\PSO(d,q))\cdot\phi({q^d-1\over q-1})$ & $d>2,\ q=p^f$ \\
\hline

\multirow{2}{*}{$\PSL(d,q){:}2$} & \multirow{2}{*}{${1\over 2df}q^{(d-1)^2\over 8} \prod_{i=1}^{d-1\over 2}(q^i+1)\cdot\phi({q^d-1\over q-1})$}
&$d>2,\ q=p^{f}$,\\
&& $f$ even \\
\hline

 $\A_r$ & $\inv((\S_2\wr\S_{r-1\over2})\cap\A_r)-\delta(r,1)$ & $r\neq 2^{2^t}+1$ \\
\hline
$\A_r$ & 
 $ \inv((\rmC_2\wr\S_{r-1\over2})\cap \A_r)-2^{2^t-2t-2}({2^{2^t}+2^{2^{t-1}}-3})-1$
 & $r=2^{2^t}+1$\\

\hline

$\S_r$ & $\inv(\rmC_2\wr\S_{r-1\over2})-\inv((\rmC_2\wr\S_{r-1\over2})\cap \A_r)-\delta(r,3)$ & \\
\hline
\end{tabular}}
\end{table}

The general case is built upon direct products of almost simple groups in some sense. 

\begin{theorem}\label{thm:rot}
Let $G$ be a finite group which has Hall rotary pairs.
Assume that $(a,z)$ is a Hall rotary pair of $G$ and $G=H\l a\r$ is a Hall factorization such that $H$ is core-free in $G$.
Then one of the following statements holds.
\begin{enumerate}[\rm (i)]
    \item $G=\l a\r{:}H=\l a\r{:}\l z\r\cong \ZZ_{k_1}\times \D_{2k_2}$,  and $\rot(G)=\Reg(G)=1$,
    where $(|a|,|H|)=(k_1k_2,2)$, and $\gcd(k_1,k_2)=1$.
  
\item \(G=(\l a_0\r\times T_1\times\dots\times T_n){.}\l c_0 c_1\dots c_n\r\leq (\l a_0\r{:}\l c_0\r)\times G_1\times\dots\times G_n\), where $G_i=T_i{.}\l c_i\r$  with $T_i=\soc(G_i)$, $T_i\ncong T_j$, and $|c_i|\leq 2$. 
Moreover, 
\begin{enumerate}[\rm(a)]
\item $H=((H_1\cap T_1)\times\dots\times (H_n\cap T_n)){.}\l c_0c_1\dots c_n\r$, 
where $(G_i,H_i,a_i)$ is a Hall triple in Table~\ref{Tab:AS-candidates};

\item letting $z_i$ be the projection of $z$ onto $G_i$ for $1\leq i\leq n$, and  $z_0$ be the projection of $z$ onto $\l a_0\r{:}\l c_0\r$, the pair $(a,z)=(a_0a_1\dots a_n,z_0z_1\dots z_n)$ is a rotary pair of $G$, and $(a_i,z_i)$ is a 
 rotary pair of $G_i$;

\item $\rot(G)=\rot(G_1)\cdots\rot(G_n)$, and $\Reg(G)=\Reg(G_1)\cdots\Reg(G_n)$,
 where $\rot(G_i)$, $\Reg(G_i)$ are as in Tables~ $\ref{tab:rot}$ and $\ref{tab:reg}$, respectively.
\end{enumerate}

   \end{enumerate}
\end{theorem}

For a Cayley map $\calM=\cayM(G,H,a,z)$, the number of vertex of $\calM$ is $|H|$ and the valency is $|a|$.
Such a map $\cal M$ is called {\it core-free} if $H$ is core-free in $G$.
Theorem~\ref{thm:rot} has a corollary regarding the orientability of $\calM$.

\begin{corollary}\label{cor:orientation}
Let $\calM=\cayM(G,H,a,z)$, where $\gcd(|H|,|a|)=1$ and $H$ is core-free in $G$.
Then either 
\begin{enumerate}[\rm(i)]
\item $\calM$ is orientable, or 

\item $\calM$ is non-orientable and bi-rotary, and $G=\l a_0\r\times T_1\times \cdots\times T_n$, where $a_0$ and $T_i$ are described in Theorem~$\ref{thm:rot}$. 
\end{enumerate}

\end{corollary}

To state our results about vertex-rotary Hall Cayley maps which are not core-free, we need the following definitions of quotients and covers of maps.
Let $\mathcal{M}=(V,E,F)$ be a $G$-vertex-rotary map.
For a normal subgroup $L\lhd G$, let $V_L, E_L, F_L$ be the sets of $L$-orbits on $V,E,F$, respectively.
Then with the natural incidence, the triple $(V_L, E_L, F_L)$ forms a map, called the {\it quotient map} of $\calM$ induced by $L$ and denoted by  $\mathcal{M}_L$.
In this case, $\mathcal{M}$ is called a {\it cover} of $\mathcal{M}_L$.

The next theorem shows that each vertex-rotary core-free Hall Cayley map has infinitely many covers which are vertex-rotary Hall Cayley maps. 

\begin{theorem}\label{regular}
For each $G$-vertex-rotary Hall Cayley map $\mathcal{M}=\cayM(G,H,a,z)$ of $H$ defined by a rotary pair $(a,z)$, there exist infinitely many elementary abelian groups $V$ such that $X=V{:}G$ and $\cayM(X,V{:}H,a,vz)$ is a cover of $\calM$ for some  $v\in V$. 
\end{theorem}

The paper is organized as follows.
After this introductory section, we state the properties of Hall rotary pairs of almost simple groups in Section \ref{sec:Hall cycles}.
Then we describe the inclusion relations between almost simple groups appearing in Table \ref{tab:1} in Section~\ref{sec:inclusions}, and deal mainly with the small groups appearing in  Theorem \ref{thm:rotG} (i)
in Section \ref{section3}. 
In Sections \ref{section5}-\ref{sec:linearreflexible}, we focus on counting Hall  rotary pairs and Hall reflexible pairs of linear groups, and
in Sections \ref{section4}-\ref{sec:refAr}, we  count Hall  rotary pairs and  Hall reflexible pairs of $\A_r$ and $\S_r$. 
In Section \ref{sec:number}, we
 complete the proof of  Theorems \ref{thm:rotG}-\ref{thm:rot} and Corollary \ref{cor:orientation}.
Finally, in Section \ref{sec:4}, we  study vertex-rotary Hall Cayley maps that are not core-free, and give the proof of Theorem~\ref{regular}. 

\section{Hall cycles and Hall rotary pairs}\label{sec:Hall cycles}

 In this section, we establish some simple but important properties of Hall cycles and Hall rotary pairs of almost simple groups.

Recall that $\calM=\CayM(G,H,a,z)$ is a $G$-vertex-rotary Cayley map of $H$ defined by a rotary pair $(a,z)$.
Let $\alpha$ be the vertex of $\cal M$ corresponding to $\l a\r$, and $G_\a$ be the stabilizer of $\a$. Then there is a group factorization
\[G=HG_\a,\ \mbox{and}\ H\cap G_\a=1.\]
In this paper, we focus on the case $\gcd(|H|,|G_\a|)=1$, so that $H$ is a Hall subgroup of $G$, and $G=H\l a\r$ is a Hall factorization.
A description is given in \cite{DiGuoLi} of such triples $(G,H,a)$ with both $H$ and $\l a\r$ core-free in $G$, which is built from Hall factorizations of the almost simple groups described in Table~\ref{tab:1}.

\begin{lemma}\label{lem:AS-candidates}
Let $G$ be an almost simple group which has a Hall triple $(G,H,a)$.
Then $(G,H,|a|)$ is a triple lying in Table~$\ref{Tab:AS-candidates}$.
\end{lemma}
{\small
\begin{table}[htpb]
\centering
\caption{\small Almost simple groups with a Hall triple $G=H\l a\r$}\label{Tab:AS-candidates}
\resizebox{\textwidth}{!}{
\begin{tabular}{|c|c|c|c|}
\hline
$G$ &  $H$ & $|a|$  & Remark \\
\hline
$\A_r,\S_r$ & $\A_{r-1},\S_{r-1}$ & $r$ & \text{$r$ prime}\\
\hline
\multirow{2}{*}{$\PSL(d,q){:}\l\phi_0\r$} & \multirow{2}{*}{$\AGL(d-1,q){:}\l\phi_0\r$} & \multirow{2}{*}{${q^d-1\over q-1}$} &  $d$ prime, $\gcd(d,q-1)=1$, \\
&&& $\phi_0$ field automorphism with $|\phi_0|\le2$\\
\hline
        $\PSL(2,11)$ & $\A_5$ & ${11}$  &\\
        $\M_{11}$ & $\M_{10}$ & ${11}$ &\\
        $\M_{23}$ & $\M_{22}$ & ${23}$  &\\
\hline
\end{tabular}
}
\label{tab:1}
\end{table}

\begin{proof}
By \cite{DiGuoLi}, the almost simple group $G$ is as listed in the first column of Table~\ref{Tab:AS-candidates}.
To complete the proof of the lemma, we may assume that $G=\PSL(d,q){:}\l\phi_0\r$, where $d$ is a prime, $\gcd(d,q-1)=1$, and $\phi_0$ is a field automorphism of $\PSL(d,q)$.
Let $T=\soc(G)=\PSL(d,q)$.
Since $G$ acting on $\calM$ is vertex-rotary, $G=\l a,z\r$ where $z$ is an involution and $(a,z)$ is a rotary pair.
As $\gcd(d,q-1)=1$, we have that $\PSL(d,q)=\PGL(d,q)$, and $a\in T$.
Write $\ov G=G/T$, $\ov a=aT$, and $\ov z=zT$. Thus $\l\phi_0\r\cong\ov G=\l\ov a,\ov z\r=\l\ov z\r\le\ZZ_2$, and so $|\phi_0|=1$ or 2, as required.
\end{proof}

Then we have the following important property regarding Hall cycles $\l a\r$.

\begin{lemma}\label{lem:G-|a|}
Let $G$ be an almost simple group which has a Hall cycle $\l a\r$.
Then all Hall cycles are conjugate and lie in $\soc(G)$.
\end{lemma}

\begin{proof}
By Lemma~\ref{lem:AS-candidates}, $G$ is listed in Table \ref{tab:1}.
Let $T=\soc(G)$, and let $\l a\r$ be a Hall cycle of $G$.
If $T\in\{\A_r,\PSL(2,11),\M_{11},\M_{23}\}$, then $\l a\r$ is a Sylow subgroup of $T$, and so all Hall cycles of $G$ are conjugate in $T$.

On the other hand, assume $T=\PSL(d,q)$ where $d$ is a prime and $\gcd(d,q-1)=1$.
Then $T=\PSL(d,q)=\PGL(d,q)$ and $\l a\r$ is a Hall subgroup of $G$.
Let $x\in G$ be such that
    $$|x|=|a|={q^d-1\over q-1}.$$

Suppose that $q^d-1$ does not have a primitive prime divisor.
By Zsigmondy's Theorem, either $d=2$ and $q+1=2^k$, or $d=6$ and $q=2$.
This contradicts the fact that $d$ is a prime and $\gcd(d,q-1)=1$.
Thus $q^d-1$ has a primitive prime divisor $s$.
Then $|a|_s=|x|_s=|G|_s$.
Let $\l a\r_s$ and $\l x\r_s$ be Sylow $s$-subgroups of $\l a\r$ and $\l x\r$, respectively.
Then $\l a\r_s$ and $\l x\r_s$ are Sylow $s$-subgroups of $G$, so that $\l a\r_s^g=\l x\r_s$ for some $g\in G$.
It is known that $\l a\r=\C_G(\l a\r_s)$ and $\l x\r=\C_G(\l x\r_s)$, see \cite[Theorem~7.3]{Huppert}.
Hence $\l a\r^g=\C_G(\l a\r_s^g)=\C_G(\l x\r_s)=\l x\r$, as required.
\end{proof}

Thus in order to determine the equivalence classes of Hall rotary pairs of an almost simple group $G$, we may fix a Hall cycle of $G$.
The  following lemma then follows, which is important in the ensuing arguments.

\begin{lemma}\label{lem:Hall-RP}
Let $G$ be an almost simple group  which has a Hall cycle $\l a\r$. Let $n=|a|$ and $m=|\N_{\Aut(G)}(\l a\r)/\l a\r|$.
Let $S=\bigcup_{\l a,z\r<G}\l a,z\r$, the union of proper subgroups $\l a,z\r$ of $G$.
Then 
\begin{enumerate}[\rm(i)]
\item each Hall rotary pair of $G$ is conjugate to $(a^i,z)$, where $\gcd(i,|a|)=1$, $z$ is an involution and $\l a,z\r=G$;

\item the number of inequivalent Hall rotary pairs of $G$ equals \[
{\phi(n)\over nm}(\inv(G)-\inv(S)).\]
\end{enumerate}
\end{lemma}

\begin{proof}
By Lemma~\ref{lem:G-|a|}, each Hall rotary pair of $G$ is conjugate to $(a^i,z)$, where $\gcd(i,n)=1$ and $\l a,z\r=G$.
There are $\phi(n)$ choices for $a^i$ with $\gcd(i,n)=1$, and for each $a^i$ there are $\inv(G)-\inv(S)$ choices for $z$ such that $\l a^i,z\r=G$.
Thus $\Delta:=\{(a^i,z)\mid \gcd(i,n)=1,\ \l a,z\r=G\}$ has cardinality equal to $\phi(n)(\inv(G)-\inv(S))$.
Since $\N_{\Aut(G)}(\l a\r)$ is of order $nm$ and acts semiregularly on the set $\Delta$, we conclude that the number of inequivalent Hall rotary pairs of $G$ is equal to ${\phi(n)\over nm}(\inv(G)-\inv(S))$.
\end{proof}


\section{Inclusions for the candidates of $G$}\label{sec:inclusions}

In order to prove our main theorems, we need to determine inclusion relations between the groups appearing in Table~\ref{tab:1}.

A factorization $G=H\l a\r$ naturally gives rise to two transitive coset actions, one on $[G:\l a\r]$ and the other on $[G:H]$. These actions can be viewed as a pair of {\it dual actions} and provide structural information to each other.
A permutation group is called a {\it c-group} if it has a transitive cyclic subgroup.
If $G=H\l a\r$ is such that $H$ is core-free and $\gcd(|H|,|a|)=1$, then $G$ is a permutation group on the set $[G:H]$ and $\l a\r$ is a regular subgroup of $G$, so that $G$ is a c-group.
Using the results on $c$-groups given in \cite{li2012cyclic}, we obtain inclusion relations among the groups in Table~\ref{tab:1}.

\begin{lemma}\label{lem:Ar-reduction}
Let $r\ge5$ be a prime, and let $G=\A_r$. 
Let $a,z\in G$ with $|a|=r$ and $|z|=2$, and let $K=\l a,z\r<G$.
Then one of the following holds:
\begin{enumerate}[\rm(i)]
\item $K=\D_{2r}$ with $r\equiv1\ (\mod 4)$;
\item $r=q+1=2^{2^t}+1$ is a Fermat prime, and $K=\PSL(2,q)$ or $\PSL(2,q){:}\l \phi_0\r$, where $\phi_0$ is the field automorphism of order $2$;
\item $r=\frac{q^d-1}{q-1}$, $K=\PSL(d,q)$, where $d\ge3$ is a prime divisor of $r-1$, and $q=p^{d^t}$ for some integer $t$;
\item $K=\PSL(2,11)$ or $\M_{11}$ with $r=11$, and $\PSL(2,11)<\M_{11}<G$;
\item $K=\M_{23}$ with $r={23}$.
\end{enumerate}
\end{lemma}

\begin{proof}
Let $G=\A_r$ naturally act on $\Ome=\{1,2,\dots,r\}$.
Note that  $\l a\r$ is a regular subgroup of $K$ acting on $\Ome$, and so $K$ is a primitive $c$-group, which is classified, see \cite[Table 1]{li2012cyclic}.
If $r=11$ or 23, then part~(iv) or (v) occurs.

Assume that $r\not=11$ or 23.
Since $\gcd(|K|/|a|,|a|)=1$, by \cite[Table\,1]{li2012cyclic}, either $K\cong \D_{2r}$, or $K\cong \PSL(d,q)$ or $\PSL(d,q){:}\l\phi_0\r$, where $r={q^d-1\over q-1}$ and $\phi_0$ is a field automorphism of order $2$.
If $K=\D_{2r}$, then $z$ normalizes $a$, and thus $z$ is a product of ${r-1\over 2}$ distinct  transpositions, so that $r\equiv 1\pmod 4$, as in part~(i).

Suppose that $K\cong \PSL(d,q)$ or $\PSL(d,q){:}\l\phi_0\r$, where 
$q=p^f$,
$r={q^d-1\over q-1}$ and $\phi_0$ is a field automorphism of order $2$.
Since $r$ is a prime, it yields that $d$ is a prime.
By Fermat's little Theorem,  $r-1=q{q^{d-1}-1\over q-1}$ is divisible by $d$.
Write $f=d^t k$ such that $\gcd(k,d)=1$.
Suppose that $k>1$.
It is easily shown that
\[
\frac{p^{d^{t+1}}-1}{p^{d^t}-1}\ \mbox{divides}\ \frac{p^{d^{t+1}k}-1}{p^{d^tk}-1},
\]
which contradicts the fact that ${q^d-1\over q-1}=r$ is prime.
Thus $f=d^t$.
In particular, if $d\ge3$, then $f=d^t$ is odd, and hence $K=\PSL(d,q)$, as in part~(iii).
On the other hand, if $d=2$, then $f=2^t$ and $q=2^f=2^{2^t}$, so that the prime $r=q+1=2^{2^t}+1$ is a Fermat prime, as in part~(ii).
 \end{proof}

The lemma shows that the primes of the form ${q^d-1\over q-1}$ are important to the proof of our main theorems, leading to the definition of Singer primes.

\begin{definition}
{\rm  A prime $r$ is called a  {\it  Singer prime} if it can be written as $$r={q^d-1\over q-1}$$ for some prime-power $q$, corresponding to the order of a Singer cycle in a finite field.}
\end{definition}

Note that a Fermat prime is a special Singer prime with $d=2$.
Singer primes constitute a  special class of so-called {\it generalized Mersenne primes}, which were originally introduced by Solinas \cite{Solinas} in 1999 and extended by Chung et al.  \cite{Chung}. For further details, we refer to \cite{Granger}.
Although generalized Mersenne primes as a whole are sparse in distribution, questions about their distribution remain open in number theory. Moreover, characterizing all possible representations of a given Singer prime, that is, determining all distinct pairs
$(q,d)$ that yield the same prime $r$, remains an interesting and unresolved problem that merits further investigation.
Lemma~\ref{lem:Ar-reduction} has an immediate corollary as follows.

 \begin{corollary}\label{lem:SingerPrime}
A Singer prime $r$ has the form ${p^{d^{t+1}}-1\over p^{d^t}-1}$, where $p$ is a prime and $d$ is a prime divisor of $r-1$.
 \end{corollary}

Now we can state the following inclusion relations between groups we concern. 
Let $G$ be a transitive permutation group acting on $\Omega_1$ and $\Omega_2$. 
Then the two actions are called {\it permutation isomorphic} if $|\Omega_1|=|\Omega_2|$ and their stabilizers are fused, namely, $(G_{\o_1})^\s=G_{\o_2}$ for some $\s\in\Aut(G)$, where $\o_1,\o_2$ are points in $\Omega_1$ and $\Omega_2$, respectively.

\begin{lemma}\label{lem:incl}
The inclusion relations between the groups in Table~$\ref{tab:1}$ are given below:
\begin{enumerate}[\rm(i)] 
\item $\PSL(2,11)<\M_{11}<\A_{11}<\S_{11}$;

\item $\M_{23}<\A_{23}<\S_{23}$; 

\item $\PSL(2,2^{2^t})<\PSL(2,2^{2^t}){:}2<\A_{r}<\S_{r}$ with $t\ge1$, $r=2^{2^t}+1$ is a Fermat prime, 
and subgroups of $\A_r$ which are isomorphic to $\PSL(2,2^{2^t})$ are conjugate in $\S_r$;

\item $\PSL(d,p^{d^t})<\A_r<\S_r$ with $d\ge3$, where $r={p^{d^{t+1}}-1\over p^{d^t}-1}$ is a Singer prime; the subgroups of $\A_r$ that are isomorphic to $\PSL(d,p^{d^t})$ are conjugate in $\S_r$. 
\end{enumerate}
\end{lemma}

\begin{proof}

By inspecting the Atlas~\cite{atlas}, we conclude that $\rm(i)$ and $\rm(ii)$ hold.

Let $G=\PGL(d,q)<\Sym(\Omega)=\S_r$, where $r={q^d-1\over q-1}$, $\gcd(d,q-1)=1$ and $d$ is a prime.
Then $G=\PSL(d,q)$.
Since $\A_r\lhd \S_r$, we have  $G\cap \A_r\lhd G$.  Noting that  $G$ is simple, we reduce $G=G\cap \A_r\leq\A_r$.
In the case where $d=2$, all subgroups of $G=\PSL(2,q)$ of index $q+1$ are conjugate in $G$.
In the case where $d\ge3$,
 a subgroup of $G=\PGL(d,q)$ of index $r={q^d-1\over q-1}$ is the stabilizer of a 1-subspace or a hyperplane.
Since the graph automorphism maps the stabilizer of a 1-subspace to the stabilizer of a hyperplane, subgroups of $G=\PGL(d,q)$ of index $r={q^d-1\over q-1}$ are conjugate in $\Aut(G)$.
In both cases,
it yields that all transitive subgroups of $\S_r$ which are isomorphic to $\PGL(d,q)$ are permutation isomorphic, and so conjugate in $\S_r$.
Finally, Lemma~\ref{lem:Ar-reduction} shows $q=2^{2^t}$ for $d=2$, and $q=p^{d^t}$ for $d\ge3$, completing the proof.
\end{proof}

We next establish a lemma that describes conjugacy classes of subgroups $K$ of $G$ which we concern.

\begin{lemma}\label{lem:double-counting}
    Let $G$ be a finite group, and let $a\in K\le G$ satisfy the conditions:
    \begin{enumerate}[\rm(i)]
        \item all cyclic subgroups of $K$ of order $|a|$ are conjugate in $K$;
       \item all cyclic subgroups of $G$ of order $|a|$ are conjugate in $G$;
        \item all subgroups isomorphic to $K$ are conjugate in $G$.
    \end{enumerate}
Then $G$ has exactly ${|K|\over |\N_G(K)|}\cdot{|\N_G(\l a\r)|\over |\N_K(\l a\r)|}$ subgroups that are isomorphic to $K$ and contain $\l a\r$.
\end{lemma}
\begin{proof}
We first observe that, since all cyclic subgroups of $G$ of order $|a|$ are conjugate by~{\rm(ii)}, there are exactly ${|G|\over |\N_G(\l a\r)|}$ cyclic subgroups of order $|a|$ in $G$.
Similarly, there are exactly ${|K|\over |\N_K(\l a\r)|}$ cyclic subgroups of order $|a|$ in $K$ by~{\rm(i)}.

We use two ways to calculate the cardinality of the set of pairs 
\[\mbox{$\Delta=\{(\l x\r,L)\mid x\in L\le G$, $|x|=|a|$, and $L\cong K\}$.}\]
On the one hand, given a subgroup $\l x\r$ with order equal to $|a|$, let
\[m(x)=|\{L<G\mid x\in L\cong K\}|.\]
By condition {\rm(ii)}, the number $m(x)$ is independent of the choice of $\l x\r$, and hence $m(x)=m(a)$ is a constant.
Since $G$ contains exactly ${|G|\over |\N_G(\l a\r)|}$ cyclic subgroups of order $|a|$, we conclude that \[|\Delta|={|G|\over |\N_G(\l a\r)|}\cdot m(a).\]

On the other hand, by~{\rm(iii)}, there are precisely ${|G|\over |\N_G(K)|}$ subgroups $L$ which are isomorphic to $K$.
Since $L$ is conjugate to $K$ by~(iii), the subgroup $L$ contains ${|K|\over |\N_K(\l a\r)|}$ cyclic subgroups of order $|a|$.
Thus $|\Delta|={|G|\over |\N_G(K)|}.{|K|\over |\N_K(\l a\r)|}$.
So
\[{|G|\over |\N_G(\l a\r)|}\cdot m=|\Delta|={|G|\over |\N_G(K)|}\cdot{|K|\over |\N_K(\l a\r)|},\]
yielding that $m={|K|\over |\N_G(K)|}\cdot{|\N_G(\l a\r)|\over |\N_K(\l a\r)|}$, as stated in the lemma.
\end{proof}

By combining Lemmas~\ref{lem:G-|a|}, \ref{lem:incl} and \ref{lem:double-counting}, we can determine proper subgroups of $\A_r$ and $\S_r$ which are generated by a pair of two elements with one of them being an involution.
This is a crucial step for calculating rotary pairs of $G$.

\section{Small groups}\label{section3}

In this section, we determine the number of inequivalent rotary pairs and reflexible pairs of the small exceptional cases appearring in Table~\ref{Tab:AS-candidates} by a series of lemmas.

\begin{lemma}\label{lem:PSL(2,11)}
Let $G=\PSL(2,11)$.
Then a Hall cycle of $G$ is of order $11$, and $G$ has $55$ involutions, and exactly $5$ inequivalent Hall rotary pairs, which are reflexible.
\end{lemma}

\begin{proof}
By Atlas~\cite{atlas}, $\ZZ_{11}{:}\ZZ_{10}$ is a subgroup of $\PGL(2,11)$. All elements of $G$ of order 11 are conjugate in $\Aut(G)=\PGL(2,11)$.
Fix an element $a\in G$ of order 11.
Then each involution $z\in G$ is such that $\l a,z\r=G$ since the proper maximal subgroups that contain $a$ are conjugate to $\ZZ_{11}{:}\Z_5$.

All involutions of $G$ are conjugate.
Let $z_0\in G$ be an involution.
Then $\C_G(z_0)=\D_{12}$ by Atlas~\cite{atlas}, and $G$ has exactly ${|G|\over|\C_G(z_0)|}=55$ involutions.
Let
\[\Delta=\{(a,z)\mid z\in G\ \mbox{with $|z|=2$ and } \l a,z\r=G\}.\]
Then each member of $\Delta$ is a Hall rotary pair of $G$, $|\Delta|=55$. Since $\C_{\Aut(G)}(a)=11$ is semiregular on $\Delta$, up to equivalence, there are exactly 5 Hall rotary pairs of $G$.

Then,
let $\s\in \Aut(G)$ be such that $a^\s=a^{-1}$.
Since $\C_G(z_0)=\D_{12}$ and $\C_G(z_0){:}\l\s\r=\D_{24}$ by Atlas~\cite{atlas}, we deduce that $[\s,z_0]=1$. Therefore, all Hall rotary pairs are reflexible.
\end{proof}

\begin{lemma}\label{lem:M11}
The Mathieu group $G=\M_{11}$ has Hall cycles of order $11$, and furthermore,
\begin{enumerate}[\rm(i)]
\item $G$ has $165$ involutions, and

\item each Hall cycle of $G$ lies in a unique subgroup isomorphic to $\PSL(2,11)$, and

\item $G$ has exactly $20$ inequivalent Hall rotary pairs, none of which is reflexible.
\end{enumerate}
\end{lemma}

\begin{proof}
All subgroups of $G$ of order 11 are conjugate, and all subgroups of $G$ which are isomorphic to $\PSL(2,11)$ are conjugate, see the Atlas \cite{atlas}.
Fix an element $a\in G$ of order 11.
Let $z\in G$ be an involution, and let $K=\l a,z\r$.
Then $K$ is a subgroup of $G=\M_{11}$ of order divisible by 11 and 2, yielding $K\cong\PSL(2,11)$ or $K\cong\M_{11}=G$, see the Atlas \cite{atlas}.

Let $K\cong\PSL(2,11)$ be a subgroup of $G$ which contains $a$.
Then $\N_G(\l a\r)=\N_K(\l a\r)=\ZZ_{11}{:}\ZZ_5$ by the Atlas \cite{atlas}, and so by Lemma~\ref{lem:double-counting}, the number of subgroups of $G$ which are isomorphic to $\PSL(2,11)$ and contain $\l a\r$ is
${|K|\over |\N_G(K)|}\cdot{|\N_G(\l a\r)|\over |\N_K(\l a\r)|}=1$.

By the Atlas \cite{atlas}, all involutions of $G$ are conjugate, and $\C_G(z)=2.\S_4$, yielding that $\inv(G)={|G|\over|\C_G(z)|}=165$, among which 55 are in $K=\PSL(2,11)$ by Lemma~\ref{lem:PSL(2,11)}, and 110 lie in $G\setminus K$.
For $1\le i\le 10$, let
\[\Delta=\{(a^i,z)\mid z\in G\ \mbox{with $|z|=2$ and } \l a,z\r=G\}.
\]
Then each member of $\Delta$ is a Hall rotary pair of $G$, and $|\Delta|=10\times110$.
Since $\N_{\Aut(G)}(\l a\r)=11{:}5$ is semiregular on $\Delta$, there are exactly ${1100\over 11\cdot 5}=20$ inequivalent Hall rotary pairs of $G$.

Finally, noticing that $a,a^{-1}$ are not conjugate in $\Aut(G)=\M_{11}$, we conclude that none of these rotary pairs is reflexible.
\end{proof}

Next, consider the alternating group $\A_{11}$, naturally acting on $\{1,2,\dots,11\}$.
We notice that a Hall cycle of $\A_{11}$ is of order $11$, and we may fix a Hall cycle $\l a\r$ with $a=(1,2,\dots,11)$.

\begin{lemma}\label{lem:A11}
Let $G=\A_{11}$.
Then $G$ has a Hall cycle of order $11$, $\PSL(2,11)<\M_{11}<\A_{11}$, and the following hold:
\begin{enumerate}[\rm(i)]
\item a Hall cycle of $G$ lies in exactly $2$ subgroups isomorphic to $\PSL(2,11)$, and exactly $2$ subgroups isomorphic to $\M_{11}$, and

\item there are exactly $1635$ inequivalent Hall rotary pairs of $G$, exactly
$155$ of which are reflexible.
\end{enumerate}
\end{lemma}

\begin{proof}
Let $X=\Aut(G)=\S_{11}$, and let $K=\M_{11}\leq G$.
By the Atlas \cite{atlas}, subgroups of $G$ which are isomorphic to $\M_{11}$ are all conjugate in $X=\S_{11}$.
It follows that $(a,K,X)$ satisfies Lemma~\ref{lem:double-counting}, and thus the number of subgroups of $X$ that are isomorphic to $K=\M_{11}$ and contain $a$ is equal to
\[{|K|\over |\N_X(K)|}\cdot{|\N_X(\l a\r)|\over |\N_K(\l a\r)|}={|11{:}10|\over|11{:}5|}=2.\]
Since $K=\M_{11}$ is simple, we have $K\leq G$.
It follows that there are exactly $2$ subgroups of $G$ which are isomorphic to $\M_{11}$ and contain $a$.
By Lemma~\ref{lem:M11}\,(ii), there are exactly $2$ subgroups of $G$ which are isomorphic to $\PSL(2,11)$ and contain $a$.
Part~(i) of the lemma is proved.

Next, we calculate the number of inequivalent Hall rotary pairs of $G$.
Let $a=(1,\dots,11)$.
As all elements of $G$ of order 11 are conjugate in $\Aut(G)=\S_{11}$, each Hall rotary pair of $G$ is equivalent to $(a,z)$ for some involution $z\in G$ such that $\l a,z\r=G=\A_{11}$.

Notice that $G=\A_{11}$ has two conjugacy classes of involutions, with representatives $x=(12)(34)$ and $y=(12)(34)(56)(78)$.
It follows that the number of involutions
\[\inv(G)={|G|\over|\C_G(x)|}+{|G|\over|\C_G(y)|}={11!\over 7!\cdot 2^3}+{11!\over 2^8\cdot3^2}.\]
Let $K_1,K_2$ be the two subgroups of $G$ which contain $a$ and are isomorphic to $\M_{11}$, refer to part~(i).
Then an involution $z\in G$ is such that $\l a,z\r<G$ if and only if $z\in K_1\cup K_2$.
It is easily shown that $K_1\cap K_2=\l a\r{:}\l b\r\cong\ZZ_{11}{:}\ZZ_5$.
It yields that $\inv(K_1\cup K_2)=\inv(K_1)+\inv(K_2)=2\inv(\M_{11})$, and a calculation shows that $\inv(G)-2\cdot\inv(\M_{11})=17985$.
Furthermore, noticing that $\C_{\Aut(G)}(\l a\r)=11$ is semiregular on the set of 17985 Hall rotary pairs $(a,z)$, we conclude that there are exactly
\[1635={17985\over 11}\]
inequivalent Hall rotary pairs of $G$.

Finally, we calculate the number of inequivalent Hall reflexible pairs.
Notice that $\s_0=(2,11)(3,10)\cdots(6,7)$ is such that $a^{\s_0}=a^{-1}$.
Let $\sigma\in\Aut(G)$ such that $a^{\s}=a^{-1}$.
Then $\sigma=a^i\s_0$.
Denote $\s_i=a^i\s_0$ for $0\leq i\leq 10$.
Since $\PGL(2,11)\not\le \Aut(G)=\S_{11}$ and all Hall rotary pairs of $\M_{11}$ are chiral, we deduce that each involution $z_i$ in $\C_G(\s_i)$ can form a reflexible pair $(a,z_i)$. 
Note that $z_i\notin\C_G(\s_j)$ for $i\neq j$. 
Let \[\Omega=\{(a,z)\mid z\in G \mbox{ with $|z|=2$ and }(a,z)\mbox{ is reflexible}\}.\]
Then $|\Omega|=11\cdot \inv(\C_G(\s_0)).$
Note that $\C_{\Aut(G)}(a)=11$ is semiregular on $\Omega$.
So there are $\inv(\C_G(\s_0))$ inequivalent Hall reflexible pairs of $G$.

As each $\s_i$ is an odd permutation, for any involution $z\in \C_{\S_{11}}(\s_i)$, exactly one of the involutions $z,\s_iz$ lies in $\A_{11}$ and the other lies in $\S_{11}\setminus \A_{11}$ except for $z=\s_i$.
Hence $\inv(\C_{G}(\s_0))={\inv(\C_{\S_{11}}(\s_0))-1\over2}.$
As $\C_{\S_{11}}(\s_0)=\S_2\wr\S_5$, we deduce that $\inv(\C_{G}(\s_0))=155$, as in  (ii).
\end{proof}

\begin{lemma}\label{lem:S11}
Let $G=\S_{11}$. Then
there are exactly $1579$ inequivalent Hall rotary pairs $(a,z)$ for $G$ with $|a|=11$,
exactly $155$ of which are reflexible.
\end{lemma}
\begin{proof}
All elements of $G$ of order 11 are conjugate in $\S_{11}$ to the element $a=(1,2,\dots,11)$.
Let $z$ be an involution of $G$, and let $K=\l a,z\r$.
If $K<\S_{11}$, then either $K=\D_{22}$, or $K\le \A_{11}$.
Denote the set of Hall rotary pairs $(a,z)$ for $G$ by
    \begin{align*}
        \Delta&=\{(a,z)\mid z\in G \mbox{ with $|z|=2$ and }\l a,z\r=G\},
    \end{align*}
Clearly, $\inv(\D_{22}\cap\A_{11})=0$, and  thus $|\Delta|= \inv(\S_{11}) - \inv(\A_{11}) - \inv(\D_{22})$.
Hence the number of inequivalent Hall rotary pairs is equal to 
${1\over 11}|\Delta|= 1579$ since $\C_{G}(\l a\r)=11$ is semiregular on $\Delta$.

Next, we consider reflexible Hall rotary pairs of $G$.
Take an involution $\s_0\in\Aut(G)=\S_{11}$ such that $a^{\s_0}=a^{-1}$, where $\s_0=(2,11)(3,10)\cdots(6,7)$.
Then the involutions
\[\s_i=a^i\s_0\ \mbox{where $0\le i\le 10$}\]
map $a$ to $a^{-1}$.
Hence an involution $z\in G$ is such that $(a,z)$ is reflexible if $z\in\C_G(\s_i)$ and $z\notin \A_{11}$ or $\D_{22}$ for some $i\in\{0,1,\dots,10\}$.
It is easily shown that
\[\mbox{$\C_{\S_{11}}(\s_i)=\S_2\wr\S_5$, and $\inv\left(\C_{\S_{11}}(\s_i)\cap \C_{\S_{11}}(\s_j)\right)=0$ for $i\not=j$}.\]
Let 
$$\Omega=\{(a,z)\mid z\in G \text{ with $|z|=2$ and }(a,z)\text{ reflexible}\}.$$ 
So we have 
\[
|\Omega|
= 11\cdot\Big(\inv(\C_{\S_{11}}(\s_0))-\inv(\C_{\A_{11}}(\s_0)) - \inv(\C_{\D_{22}}(\s_0))\Big).
\]
As each $\s_i$ is an odd permutation, for any involution $z\in \C_{\S_{11}}(\s_i)$, exactly one of the involutions $z,\s_iz$ lies in $\A_{11}$ and the other lies in $\S_{11}\setminus \A_{11}$ except for $z=\s_i$.
Hence 
$$\inv(\C_{\S_{11}}(\s_0))-\inv(\C_{\A_{11}}(\s_0))={\inv(\C_{\S_{11}}(\s_0))+1\over2}.$$
It is clear that $\inv(\C_{\D_{22}}(\s_0))=1$. 
Thus $|\Omega|=11\cdot{{\inv(\C_{\S_{11}}(\s_0))-1\over2}}=11\cdot155$. 
Since $\C_{\Aut(G)}(\l a\r)=11$ is semiregular on $\Omega$, it follows that there are $155$ inequivalent Hall reflexible pairs of $G$.
\end{proof}



\begin{lemma}\label{lem:M23}
The Mathieu group $G=\M_{23}$ has $3\cdot5\cdot11\cdot23$ involutions, and exactly $330$ inequivalent Hall rotary pairs, none of which is reflexible.
\end{lemma}

\begin{proof}
It is known that $\Aut(G)=G$, that the elements $a$ of order 23 form two conjugacy classes with $a$ and that $a^{-1}$ non-conjugate and $\N_G(\l a\r)=23{:}11$, see the Atlas \cite{atlas}.
Fix an element $a$ of order 23, and let
\[\Delta=\{(a,z),(a^{-1},z)\mid z\in G \mbox{ with $|z|=2$ and }\l a,z\r=G\}.\]
Then each member of $\Delta$ is a Hall rotary pair of $G$.
Since all involutions of $G$ are conjugate, for an involution $z_0$, the centralizer $|\C_G(z_0)|=2^7\cdot3\cdot7$, and so there are $3\cdot5\cdot11\cdot23={|G|\over|\C_G(z_0)|}$  involutions.
Thus $|\Delta|=2\cdot3\cdot5\cdot11\cdot23$.
Since $\C_G(\l a\r)=23$ is semiregular on $\Delta$, there are exactly ${|\Delta|\over 23}=330$ inequivalent Hall rotary pairs of $G$.
For a rotary pair $(a,z)$, as $a$ and $a^{-1}$ are not conjugate under $\N_{\Aut(G)}(\l a\r)=23{:}11$, the Hall rotary pair $(a,z)$ is not reflexible.
\end{proof}


\begin{lemma}\label{lem:A23}
For $G=\A_{23}$, the following hold:
\begin{enumerate}[\rm(i)]
\item each Hall cycle lies in exactly $2$ subgroups isomorphic to $\M_{23}$,

\item there are exactly $72633517935$ inequivalent Hall rotary pairs of $G$, exactly $2611103$ of which are reflexible.
\end{enumerate}

\end{lemma}

\begin{proof}
As before, we may assume $a=(1,2,\dots,23)$.
Then $\N_G(\l a\r)=\ZZ_{23}{:}\ZZ_{11}$, and each involution $z\in G$ is such that $\l a,z\r=\M_{23}$ or $\A_{23}$.
Let $X=\Aut(G)=\S_{23}$, and let $K=\M_{23}$.
By the Atlas \cite{atlas}, subgroups of $G$ which are isomorphic to $\M_{23}$ are all conjugate in $X=\S_{23}$.
It follows that $(a,K,X)$ satisfies Lemma~\ref{lem:double-counting}, and thus the number of subgroups of $X$ that are isomorphic to $K=\M_{23}$ and contain $a$ is equal to
\[{|K|\over |\N_X(K)|}\cdot{|\N_X(\l a\r)|\over |\N_K(\l a\r)|}={|23{:}22|\over|23{:}11|}=2.\]
Since $K=\M_{23}$ is simple, we have $K\leq G$.
It follows that there are exactly $2$ subgroups {\color{blue}of $G$} which are isomorphic to $\M_{23}$ and contain $a$.

Let $K_1,K_2$ be the two subgroups of $G$ which contain $a$ and are isomorphic to $\M_{23}$, refer to part~(i).
Then an involution $z\in G$ is such that $\l a,z\r<G$ if and only if $z\in K_1\cup K_2$.
It is easily shown that $K_1\cap K_2=\l a\r{:}\l b\r\cong\ZZ_{23}{:}\ZZ_{11}$.
It yields that $\inv(K_1\cup K_2)=\inv(K_1)+\inv(K_2)=2\inv(\M_{23})$.

As each involution $z\in G$ is such that either $\l a,z\r=\M_{23}$ or $\l a,z\r=\A_{23}$, the $(\inv(\A_{23})-2\inv(\M_{23}))$ involutions $z$ of $G$ are such that $\l a,z\r=G$, namely,
\[\Delta=\{(a,z)\mid z\in G \mbox{ with $|z|=2$ and }\l a,z\r=G\},\]
has cardinality equal to $\inv(\A_{23})-2\inv(\M_{23})$.

Since $\C_{\Aut(G)}(a)=23$ is semiregular on $\Delta$, we conclude that there are exactly ${1\over23}(\inv(\A_{23})-2\inv(\M_{23}))$
inequivalent Hall rotary pairs of $G$.

By Lemma~\ref{lem:M23}, we have $\inv(\M_{23})=3\cdot5\cdot11\cdot23$. For $\inv(\A_{23})$, 
since each involution in $\S_{23}$ is a product of $k$ disjoint transpositions for some even $k$ with $1\le k\le 11$, we calculate $\inv(\A_{23})$ by classifying involutions according to how many transpositions there are.
Let $I_k$ be the set of involutions which are products of $k$ disjoint transpositions.
Then $I_k$ is a conjugacy class in $\S_{23}$.
The involution $z=(12)\dots(2k-1,2k)$ belongs to $I_k$, and
\[\C_{\S_{23}}(z)=\C_{\Sym\{1,\dots,2k\}}(z)\times\Sym\{2k+1,\dots,r\}=2^k.\S_k\times\S_{23-2k}.\]
Thus $|I_k|={23!\over 2^k\cdot k!\cdot (23-2k)!}$, and so
\[\inv(\A_{23})=|I_2|+|I_4|+\dots+|I_{10}|=\sum_{k=1,k \text{ even}}^{10}{23!\over(23-2k)!\cdot2^k\cdot k!}.\]
Therefore, $$|\Delta|=\inv(\A_{23})-2\cdot\inv(\M_{23})=167057912505,$$
and the number of inequivalent Hall rotary pairs of $G$ follows.

Finally, we calculate the number of inequivalent reflexible pairs.
Pick an involution
\[\s_0=(2,23)(3,22)\cdots(12,13).\]
Then $a^{\s_0}=a^{-1}$, and $\C_{\S_{23}}(\s_0)\cong \S_2\wr\S_{11}$.
Let $\sigma\in\Aut(G)$ such that $a^{\s}=a^{-1}$.
Then $\sigma=a^i\s_0$.
Denote $\s_i=a^i\s_0$ for $0\leq i\leq 22$.
Since all Hall rotary pairs of $\M_{23}$ are chiral, we deduce that each involution $z_i$ in $\C_G(\s_i)$ can form a Hall reflexible pair $(a,z_i)$. Note that $z_i\notin\C_G(\s_j)$ for $i\neq j$. 
Let \[\Omega=\{(a,z)\mid z\in G \mbox{ with $|z|=2$ and }(a,z)\mbox{ is reflexible}\}.\]
Then $|\Omega|=23\cdot \inv(\C_{G}(\s_0)).$
Note that $\C_{\Aut(G)}(a)=11$ is semiregular on $\Omega$.
So there are ${23\cdot \inv(\C_{G}(\s_0))\over23}=\inv(\C_{G}(\s_0))$ inequivalent Hall reflexible pairs of $G$.

Since $\s_0$ is an odd permutation, for any involution $z\in \S_2\wr\S_{11}$, precisely one of the involutions $z,\s_0z$ lies in $G=\A_{23}$ except for $z=\s_0$.
Hence $\inv(\C_{G}(\s_0))={\inv(\C_{\S_{23}}(\s_0))-1\over2}$.
Calculation shows that $\inv(\C_G(\s_0))$ equals $2611103$.
Thus there are $2611103$ inequivalent Hall reflexible pairs of $G$.
\end{proof}

\section{Linear groups}\label{section5}

In this section, we determine Hall rotary pairs of linear groups $G$ listed in the third row of Table \ref{tab:1}.
Let $T=\PSL(d,q)$, where $d$ is a prime, $q=p^f$ with $p$ prime, and $\gcd(d,q-1)=1$.
Then either $G=T$ or $G=T{:}\l\phi_0\r$, where $\phi_0$ is a filed automorphism of $T$ with order $2$.
In addition, $\PGL(d,q)=\PSL(d,q)=T$, and
\[\Aut(T)=\Aut(G)=\begin{cases}
    \PGammaL(d,q){:}\l\g\r&\textup{ if } d>2,\\
    \PGammaL(2,q)&\textup{ if }d=2,
\end{cases}\]
where $\g$ is the graph automorphism of $T$.
In this case, a Hall cycle is a so-called Singer cycle.
As shown in Lemma~\ref{lem:G-|a|}, we may fix a Singer cycle $\l a\r$.
Let 
\[n=|a|={q^d-1\over q-1}.\]

\begin{lemma}\label{lem:PSL-1}
 With the notation defined above, each involution $z\in G$ is such that $\l a,z\r=\ZZ_n{:}\ZZ_2$, $T$ or $G$. 
\end{lemma}

\begin{proof}
Since $a\in T$ by Lemma~\ref{lem:G-|a|}, the factor group $\l a,z\r T/T\cong\l\ov z\r\le\ZZ_2$,  where $\ov z=zT$.
If $z$ normalizes $\l a\r$,  then
$\l a,z\r=\l a\r{:}\l z\r=\ZZ_n{:}\ZZ_2$.
On the other hand, if $z$ does not normalize $\l a\r$, then as $d$ is a prime,  by the structures of maximal subgroups of $G$, see \cite{KL}, $\l a,z\r=\PSL(d,q)$ or $\PSL(d,q){:}2$.
\end{proof}

We next treat the cases $d=2$ and $d>2$ separately.

\begin{lemma}\label{lem:PSL-d=2}
Let $G=\PSL(2,q)$, where $q=2^f$.
Then the number of inequivalent Hall rotary pairs of $G$ equals
${\phi(q+1)\over 2(q+1)f}\left(\inv(G)-(q+1)\right)={1\over 2f}\phi(q+1)(q-2)$, and each of them is reflexible.
\end{lemma}

\begin{proof}
In this case, $n=|a|=q+1$.
For any involution $z\in G$, by Lemma \ref{lem:PSL-1}, either $z$ normalizes $\l a\r$, or $\l a,z\r=G$ .
Since $\N_G(\l a\r)=\D_{2n}$, there are exactly $n$ involutions $z\in G$ such that $\l a,z\r<G$, and the other $\inv(G)-n$ involutions $z$ are such that $\l a,z\r=G$.

By Lemma~\ref{lem:Hall-RP}, each Hall rotary pair of $G$ is conjugate to $(a^i,z)$, where $\gcd(i,n)=1$ and $\l a,z\r=G$.
There are $\phi(n)$ choices for $a^i$ with $\gcd(i,n)=1$, and for each $a^i$ there are $\inv(G)-n$ choices for $z$ such that $\l a^i,z\r=G$.
Thus the set
$$\Delta:=\{(a^i,z)\mid \gcd(i,n)=1,\ \l a,z\r=G\}$$ 
has cardinality equal to $\phi(n)(\inv(G)-n)$.
Recall that $\N_{\Aut(G)}(\l a\r)=\l a\r{:}\ZZ_{2f}$.
Since $\N_{\Aut(G)}(\l a\r)$ acts semiregularly on the set $\Delta$, we conclude that the number of inequivalent Hall rotary pairs of $G$ equals ${|\Delta|\over n\cdot2f}={\phi(n)\over n\cdot2f}(\inv(G)-n)$.
Since all involutions of $G$ are conjugate and $\C_G(z)=q$, we conclude that $\inv(G)=q^2-1$.
Hence $\inv(G)-n=(q^2-1)-(q+1)=(q+1)(q-2)$, and the number of inequivalent Hall rotary pairs follows.

There are $q+1$ Sylow 2-subgroups $P_j$ with $0\le j\le q$, which are elementary abelian.
It follows that the $q+1$ involutions $a^j\s\in\N_G(\l a\r)=\D_{2(q+1)}$ lie in different Sylow 2-subgroups.
Relabelling if necessary, we may assume that $a^j\s\in P_j$.
It yields that each Hall rotary pair is reflexible, completing the proof.
\end{proof}

We remark that it was first proved by Jones \cite{jones2022} that rotary pairs of $\PSL(2,q)$ are reflexible.

\begin{lemma}\label{lem:PSL(2,q).2}
Let $G=\PSL(2,q){:}\l\phi_0\r=\PSL(2,q){:}\ZZ_2$, where $q=2^{f}$.
Then $n=|a|=q+1$, and the number of inequivalent Hall rotary pairs of $G$ is equal to ${\phi(n)\over2nf}(\inv(G)-\inv(T))={\phi(n)\over2f}q^{1/2}$.
\end{lemma}

\begin{proof}
In this case, the Hall cycle $\l a\r<T=\PSL(2,q)$, and thus a Hall rotary pair $(a,z)$ is such that $z$ is an involution of  $G\setminus T$.
 It then follows from Lemma \ref{lem:PSL-1} that  for each involution $z\in G\setminus T$, either  $\l a,z\r=\ZZ_n{:}\ZZ_2$, or $\l a,z\r=G$.

Since $\N_T(\l a\r)=\D_{2n}$ and $\N_G(\l a\r)=\ZZ_n{:}\ZZ_4$, all of the involutions on $\N_G(\l a\r)$ lie in $T$.
Hence every involution $z\in G\setminus T$ is such that $\l a,z\r=G$, and so $(a^i,z)$ with $\gcd(i,n)=1$ is a Hall rotary pair.
There are $\phi(n)$ choices for $a^i$, and for each $a^i$ with $\gcd(i,n)=1$ there are $\inv(G)-\inv(T)$ choices for $z$.
Thus each Hall rotary pair of $G$ is conjugate to one of the $\phi(n)(\inv(G)-\inv(T))$ Hall rotary pairs in the set $\Delta:=\{(a^i,z)\mid \gcd(i,n)=1,\ z\in G\setminus T\}$.
Obviously, $\N_{\Aut(G)}(\l a\r)=\l a\r{:}\ZZ_{2f}$ is semiregular on the set $\Delta$, and so
the number of Hall rotary pairs of $G$ is equal to ${\phi(n)\over 2nf}(\inv(G)-\inv(T))$.

Finally, we calculate the number of involutions $\inv(G\setminus T)=\inv(G)-\inv(T)$.
Since field automorphisms of order $2$ are conjugate by \cite[Table B.3]{Burness16},
 all involutions of $G\setminus T$ are conjugate to $\phi_0$.
We then have $\inv(G\setminus T)={|G|\over |\C_G(\phi_0)|}$.
Let $q_0=q^{1/2}$.
Then $\C_G(\phi_0)=\PSL(2,q_0)\times\l\phi_0\r$,
and hence 
\[\inv(G\setminus T)={|G|\over |\C_G(\phi_0)|}={|\PSL(2,q)|\times 2\over|\PSL(2,q_0)|\times 2}={q(q^2-1)\over q_0(q_0^2-1)}=q^{1/2}(q+1).\]
This completes the proof.
\end{proof}

\begin{lemma}\label{lem:PSL-d>2}
Let $G=\PSL(d,q)$, and let $n=|a|={q^d-1\over q-1}$, where $d>2$ is a prime and $q=p^f$.
Then the number of inequivalent Hall rotary pairs of $G$ is equal to ${\phi(n)\over 2ndf}\inv(\PSL(d,q))$.
\end{lemma}

\begin{proof}
Let $z\in G$ be an involution. By Lemma \ref{lem:PSL-1}, either $\l a,z\r=G$, or $z$ normalizes $\l a\r$.
For the latter case,  we have that $\l a,z\r\le\N_G(\l a\r)=\ZZ_n{:}\Z_d$, which is not possible since $d$ is odd.
Thus each of the $\inv(G)$ involutions $z$ is such that $\l a,z\r=G$.

By Lemma~\ref{lem:Hall-RP}, each Hall rotary pair of $G$ is conjugate to $(a^i,z)$, where $\gcd(i,n)=1$ and $\l a,z\r=G$.
There are $\phi(n)$ choices for $a^i$ with $\gcd(i,n)=1$, and for each $a^i$ there are $\inv(G)$ choices for $z$ such that $\l a^i,z\r=G$.
Thus the set
$$\Delta:=\{(a^i,z)\mid \gcd(i,n)=1,\ \l a,z\r=G\}$$ 
has cardinality equal to $\phi(n)\inv(G)$.
Since $\N_{\Aut(G)}(\l a\r)=\l a\r{:}(\ZZ_{df}\times\ZZ_2))$, and 
$\N_{\Aut(G)}(\l a\r)$ acts semiregularly on the set $\Delta$, we conclude that the number of inequivalent Hall rotary pairs of $G$ equals ${|\Delta|\over n\cdot2df}={\phi(n)\over 2ndf}\inv(G)$.
\end{proof}

Next, we count the number of inequivalent rotary pairs of $G$, where $G=\PSL(d,q){:}\l\phi_0\r$ with $d$ odd prime and $|\phi_0|=2$.

\begin{lemma}\label{4.4}
Let $T=\PSL(d,q)$, and $G=T{:}\l\phi_0\r=\PSL(d,q){:}\ZZ_2$, and let $n=|a|={q^d-1\over q-1}$, where $d>2$ is a prime and $q=p^f$.
Then the number of inequivalent Hall rotary pairs of $G$ is equal to 
\[{\phi(n)\over 2ndf}\Big(\inv(G)-\inv(T)-\inv(\N_G(\l a\r))\Big)={\phi(n)\over 2ndf}\left(q_0^{d(d-1)/2}\prod_{i=2}^d(q_0^i+1) - \frac{q_0^d+1}{q_0+1}\right),\]
where $\inv(\N_G(\l a\r))=\frac{q_0^d+1}{q_0+1}$ with $q_0=q^{1/2}$.
\end{lemma}

\begin{proof}
Since $\gcd(d,q-1)=1$, it follows that $a\in T$ by Lemma~\ref{lem:G-|a|}.
Thus an involution $z$ is such that $\l a,z\r=G$ if and only if $z\in G\setminus (T\cup\N_G(\l a\r))$.
Let
\[\Delta=\{(a^i,z)\mid \gcd(i,n)=1,\ z\in G \mbox{ with $|z|=2$ and }\l a^i,z\r=G\}.\]
There are exactly $\phi(n)$ choices for $a^i$ with $\gcd(i,n)=1$.
By \cite[Theorem 7.3]{Huppert}, we have $\N_T(\l a\r)=\ZZ_n{:}\ZZ_d$, and 
\[
K:=\N_{G}(\l a\r)= \l a\r{:}\l b\r\cong \ZZ_n{:}\ZZ_{2d}.
\]
Then $K\cap T=\N_T(a)$ does not contain involutions, and so $\inv(T\cup K)=\inv(T)+\inv(K)$.
By Lemma \ref{lem:PSL-1}, any involution $z\in G$ is such that either $\l a,z\r\le K$, or $\l a,z\r= T$ or $G$. 
Thus there are exactly $\inv(G)-\inv(K\cup T)$ choices for $z$ such that $\l a^i,z\r=G$, and so 
\[
|\Delta|=\phi(n)(\inv(G)-\inv(K\cup T))=\phi(n)(\inv(G) - \inv(K) - \inv(T)).
\]

Obviously, $\N_{\Aut(G)}(\l a\r)=\l a\r{:}(\Z_{df}\times\Z_2)$ has order $2ndf$, and acts semiregularly on $\Delta$, and so the number of inequivalent Hall rotary pairs is equal to 
$${|\Delta|\over|\N_{\Aut(G)}(\l a\r)|}={\phi(n)\over 2ndf}\left(\inv(G)-\inv(T)-\inv(K)\right).$$

We claim
\[
\inv(K)= \frac{q_0^d+1}{q_0+1}.
\]
Since $|a|$ is odd, any involution in $K$ must have the form $a^ib^d$. Using  $a^{b^d}=a^{q_0^d}$ (see \cite{Low-dim-book}), the condition
$(a^i)^{b^d} = a^{-i}$ is equivalent to
\[
i(q_0^d+1) \equiv 0 \pmod{\frac{q^d-1}{q-1}}.
\]
Since $\gcd(q_0^d+1,\frac{q^d-1}{q-1})
=\frac{q_0^d+1}{q_0+1}$, there are exactly $\frac{q_0^d+1}{q_0+1}$ solutions, as  claimed.
Noting that $$\inv(G)=\inv(T)+{|\PGammaL(d,q)|\over f|\PGL(d,q_0)|},$$
 as field automorphisms of order 2 are conjugate by \cite[Table B.3 ]{Burness16}, it yields
\[
|\Delta|= \phi(n)(\inv(G) - \inv(T) - \inv(K))
= \phi(n)\cdot\left(q_0^{d(d-1)/2}\prod_{i=2}^d(q_0^i+1) - \frac{q_0^d+1}{q_0+1}\right).
\]
Therefore, the number of inequivalent rotary pairs of $G$ is as given in the lemma.
\end{proof}

\section{Reflexible pairs of linear groups}\label{sec:linearreflexible}

In this section, we 
determine the number of inequivalent Hall reflexible pairs $(a,z)$ of $G$.

Let $T=\PSL(d,q)$, and let $G=T$ or $G=T{:}\l \phi_0\r$, where $d$ is a prime, $q=p^f$ with $p$ prime, $\gcd(d,q-1)=1$, and $\phi_0$ is a field automorphism of $T$ with order $2$.
Fix a Singer cycle $\l a\r$, and let \[n=|a|={q^d-1\over q-1}.\]

 We first deal with the case $d=2$. Note that we have shown that each Hall rotary pair of $\PSL(2,q)$ with $q=2^f$ is reflexible in Lemma \ref{lem:PSL-d=2}.
\begin{lemma}\label{lem:PSL(2,q).2-ref}
Let $G=\PSL(2,q){:}\l\phi_0\r=\PSL(2,q){:}\ZZ_2$, where $q=2^{f}$ for some even integer  $f$.
Then  $n=|a|=q+1$, and the number of inequivalent Hall reflexible pairs of $G$ is equal to ${\phi(n)\over2f}q^{1\over2}$.
\end{lemma}
\begin{proof}
Let $(a,z)$ be a rotary pair of $G$, so that $n=q+1$ and $z\in G\setminus T$.
Then $z=t\phi_0$, where $t\in T$.
Assume that an involution $\s\in\Aut(G)$ is such that $(a,z)^\s=(a^{-1},z)$.
Then $\s\in \N_{\Aut(G)}(\l a\r)=\l a\r{:}\l b\r\cong \ZZ_n{:}\ZZ_{2f}$, where $a^b=a^2$ and $|b|=2f$. 
Since $|\s|=2$, we have $\s=a^jb^{f}$ for some $j$ with $0\le j\le q$, and  $\l a,\s\r= \D_{2(q+1)}<T$.

By Lemma~\ref{lem:Hall-RP}, each Hall reflexible pair of $G$ is conjugate to $(a^i,z)$, where $\gcd(i,n)=1$ and $z$ is an involution of $G\setminus T$ such that $(a,z)$ is a rotary pair of $G$ which is reflexible. 
Let \[
\Omega=\{(a^i,z)\mid \gcd(i,n)=1,\  (a,z) \textup{ is  reflexible} \}.
\]
Suppose $z$ is an involution of $G\setminus T$.
Then by the proof of Lemma \ref{lem:PSL(2,q).2}, there must be  $\l a,z\r=G$. 
Denote $\s_j=a^jb^{f}$ for $0\le j\le q$. 
It follows that  $(a^i,z)$ belongs to $\Omega$ if and only if $\gcd(i,n)=1$ and  $z\in \C_G(\s_j)\setminus \C_T(\s_j)$ for some $0\le j\le q$.
If $z\in \C_G(\s_j)\cap \C_G(\s_{j'})$ for some $0\le j, j'\le q$, then  $(a,z)^{\s_j\s_{j'}^{-1}}=(a,z)$, and thus $\s_j\s_{j'}^{-1}=1$ since $\l a,z\r=G$. 
Moreover,  noting that each  $\s_j$ is conjugate to $\s_0$ in $\Aut(G)$, see \cite[Table B.3]{Burness16}, we have $\C_G(\s_j)\cong \C_G(\s_0)$ and $\inv(\C_G(\s_j))=\inv(\C_G(\s_0))$.  
It thus follows that
\begin{align*}
    |\Omega|=&\phi(n)\cdot \inv\Big(\cup_{0\le j\le q}(\C_G(\s_j)\setminus\C_T(\s_j))\Big)\\
    =&\phi(n)\cdot \Big(\sum_{0\le j\le q}\inv(\C_G(\s_j)\setminus \C_T(\s_j))\Big)\\
    =&\phi(n)(q+1)\cdot \Big(\inv\big(\C_G(\s_0)\setminus\C_T(\s_0)\big)\Big).
\end{align*}

Note that   $\s_0=b^f\in T$ and $\s_0$ lies in a  Sylow 2-subgroup $P$ of $T$, which  is an elementary abelian $2$-group of order $q$. 
We then have $\C_T(\s_0)=P$.
Since $b^{f/2}\in \C_G(\s_0)$ and $b^{f/2}\notin T$, we have
$\C_G(\s_0) /\C_T(\s_0)\cong \ZZ_2$. 
It follows that $\C_G(\s_0)=\C_T(\s_0){.}\ZZ_2= P{.}\Z_2$, and  $\C_G(\s_0)$ is a Sylow $2$-subgroup of $G$. Since
$\phi_0$ lies in a Sylow $2$-subgroup of $G$,
there must be $\phi_0^g\in \C_G(\s_0)$ for some $g\in G$. Set $\phi_0'=\phi_0^g$.
 Then we have $\C_G(\s_0)=\C_T(\s_0){:}\l \phi_0'\r\cong E_q{:}\ZZ_2$.

Since $\PSL(2,q)\cong \SL(2,q)$,
each element in $P$ can be regarded as an upper triangular matrix and $\phi_0'$ maps each entry of the matrix to its $q^{1\over 2}$-th power.
Hence involutions in $\C_G(\s)\setminus \C_T(\s)$ have the form 
$\begin{pmatrix}
  1& x\\
  0&1
\end{pmatrix}\phi_0'$,
where
$\begin{pmatrix}
  1& x^{q^{1\over 2}}+x\\
  0&1
\end{pmatrix}=I_2$.
It yields $x^{q^{1/2}}=x$, and
 the number of these solutions is $q^{1\over 2}$.
Therefore, the number of involutions in $\C_G(\s_i)\setminus \C_T(\s_i)$ is $q^{1\over 2}$.
We then deduce  that 
$$|\Omega|=\phi(n)q^{1\over 2}(q+1).$$
As $\N_{\Aut(G)}(\l a\r)$ is semi-regular on $\Omega$,  the number of inequivalent Hall reflexible pairs is equal to ${\phi(n)\over2f}q^{1\over 2}$.



\end{proof}

\begin{lemma}\label{lem:PSL-d>2-ref}
Let $G=\PSL(d,q)$, and let $n=|a|={q^d-1\over q-1}$, where $d>2$ is a prime and $q=p^f$. 
Then the number of inequivalent Hall reflexible pairs of $G$ is equal to 
\[
\Reg(G)={\phi(n)\over {2df}}\inv(\C_G(\gamma))=
\begin{cases}
    \frac{\phi(n)}{2df}\cdot \inv(\PSO(d,q)), \textup{ if }  q \textup{ is odd},\\
    \frac{\phi(n)}{2df}\cdot \inv(\Sp(d-1,q)), \textup{ if }  q \textup{ is even},
\end{cases}
\] where $\gamma$ is a graph automorphism of $G$.
\end{lemma}

\begin{proof}
Let $(a,z)$ be a reflexible pair of $G$.
Suppose that $\s\in \Aut(G)$ is such that $(a,z)^\s=(a^{-1},z)$. Then $\s\in \N_{\Aut(G)}(\l a\r)$. 
By \cite{KL}, we have $$\N_{\Aut(G)}(\l a\r)=\l a\r{:}(\l b\r\times\l \gamma\r)\cong \ZZ_{n}{:}(\Z_{df}\times \Z_2),$$
where $a^b=a^p$ and $\gamma$ is a graph automorphism such that  $a^\gamma=a^{-1}$.
Then $\s=a^lb^j\gamma^k$ for some $0\le l\le n-1$, $0\le j\le df-1$, and $0\le k\le 1$.
Since $a^\sigma=a^{-1}$, we deduce that $j=0$ and $\s=a^l\gamma$ is a graph automorphism for some $l$. 
Denote $\s_l=a^l\gamma$ for $0\leq l\leq n-1$. By \cite[Table B.3]{Burness16}, those $\s_l$ are conjugate, since $d$ is a prime. Thus $\C_G(\s_l)\cong \C_G(\gamma)$ and $\inv(\C_G(\s_l))=\inv(\C_G(\gamma))$.

By Lemma~\ref{lem:Hall-RP}, each Hall reflexible pair of $G$ is conjugate to $(a^i,z)$, where $\gcd(i,n)=1$.
There are $\phi(n)$ choices for $a^i$ with $\gcd(i,n)=1$. 
Let \[
\Omega=\{(a^i,z)\mid \gcd(i,n)=1,\  (a^i,z) \textup{ is  reflexible} \}.
\]
Since $\N_G(\l a\r)=\l a\r{:}\ZZ_d$ is of odd order, by Lemma \ref{lem:PSL-1},  each involution $z$ in $G$ is such that $\l a,z\r=G$.
Note that $(a,z)^x=(a,z)^y$ for $x,y\in \Aut(G)$ if and only if $x=y$ since $\l a,z\r=G$. So we conclude that  $\inv(\C_G(\s_l)\cap \C_G(\s_{l'}))=0$ for $l\neq {l'}$.
Hence $$|\Omega|=\phi(n)n\inv(\C_G(\gamma)). $$
Since $\N_{\Aut(G)}(\l a\r)=\Z_n{:}(\Z_{df}\times 2)$ is semi-regular on $\Omega$, we conclude that the number of inequivalent Hall reflexible pairs of $G$ is equal to $${\phi(n)\over 2df}\inv(\C_G(\gamma)).$$
By \cite[Proposition 3.2.11]{Burness16},
$\C_G(\gamma)\cong{\rm PGO}(d,q)=\PSO(d,q)$ for odd $q$, and $\C_G(\gamma)\cong\Sp(d-1,q)$ for even $q$. 
Therefore, the results follow.
\end{proof}

\begin{lemma}\label{lem:PSL.2-d>2-ref}
Let $T=\PSL(d,q)$, and $G=T{:}\l\phi_0\r=\PSL(d,q){:}\ZZ_2$, and let $n=|a|={q^d-1\over q-1}$, where $d>2$ is a prime and $q=p^{f}$.
Then the number of inequivalent Hall reflexible pairs of $G$ is equal to 
\[{\phi(n)\over 2df}\Big(\inv(\C_{G}(\gamma))-\inv(\C_{T}(\gamma))\Big)={\phi(n)\over 2df}q^{(d-1)^2\over 8} \prod_{i=1}^{d-1\over 2}(q^i+1),\]
where $\gamma$ is a graph automorphism of $T$.
\end{lemma}
\begin{proof}
Let $(a,z)$ be a reflexible pair of $G$,
and let $\s\in \Aut(G)$ be such that $(a,z)^\s=(a^{-1},z)$.
It is similar to the proof of Lemma~\ref{lem:PSL-d>2-ref}, we conclude that $\s=a^l\gamma$ for some $0\le l\le n-1$. 
Denote $\s_l=a^l\gamma$. By \cite[Table B.3]{Burness16}, those $\s_l$ are conjugate, since $d$ is a prime. Thus $\inv(\C_G(\s_l))=\inv(\C_G(\gamma))$.

By Lemma~\ref{lem:Hall-RP}, each Hall reflexible pair of $G$ is conjugate to $(a^i,z)$, where $\gcd(i,n)=1$.
There are $\phi(n)$ choices for $a^i$ with $\gcd(i,n)=1$. 
Let \[
\Omega=\{(a^i,z)\mid \gcd(i,n)=1,\  (a^i,z) \textup{ is  reflexible} \}.
\]
Since $\N_G(\l a\r)=\l a\r{:}\ZZ_d$ is of  odd order, by Lemma \ref{lem:PSL-1},  either $\l a,z\r=T$, or $G$, where $|z|=2$. Thus $(a^i,z)\in \Omega$ if and only if $\gcd(i,n)=1$ and  $z\in \C_G(\s_l)\setminus\C_T(\s_l)$ for some $0\leq l\leq n-1$.

Note that $(a,z)^x=(a,z)^y$ for $x,y\in \Aut(G)$ if and only if $x=y$ since $\l a,z\r=G$. So we conclude that  $\inv(\C_G(\s_l)\cap \C_G(\s_{l'}))=0$ for $l\neq l'$.
Hence we have $$|\Omega|=\phi(n)n(\inv(\C_G(\gamma))-\inv(\C_T(\gamma))).$$ Since $\N_{\Aut(G)}(\l a\r)=\Z_n{:}(\Z_{df}\times 2)$ is semi-regular on $\Omega$, we conclude that the number of inequivalent Hall reflexible pairs of $G$ is equal to $${\phi(n)\over 2df}(\inv(\C_G(\gamma))-\inv(\C_T(\gamma))).$$

One can directly check that $\C_{G}(\gamma)=\C_{T}(\gamma){:}\l\phi\r\cong\C_{T}(\gamma){:}2$.
When $q$ is odd, $\C_{T}(\gamma)\cong\PSO(d,q)$.
It follows that
\begin{align*}
  \inv(\C_{G}(\gamma))-\inv(\C_{T}(\gamma))=&\inv(\PSO(d,q){:}2)-\inv(\PSO(d,q)).
\end{align*}
Observe that $\inv(\PSO(d,q){:}2)-\inv(\PSO(d,q))$ is the number of involutions that are field automorphisms of $\PSO(d,q){:}2$. Since all such automorphisms are conjugate, see \cite[Section 3.5.5]{Burness16}, we conclude that
$$\inv(\C_G(\gamma))-\inv(\C_T(\gamma))=\frac{|\PSO(d,q){:}2|}{|\PSO(d,q_0){:}2|}
=q^{(d-1)^2\over 8}\prod_{i=1}^{d-1\over 2}(q^i+1).$$
When $q$ is even, $\C_{T}(\gamma)=\Sp(d-1,q)$.
Similarly, we conclude that
\begin{align*}
  \inv(\C_{G}(\gamma))-\inv(\C_{T}(\gamma))&=\inv(\Sp(d-1,q){:}2)-\inv(\Sp(d-1,q))\\
  &=\frac{|\Sp(d-1,q){:}2|}{|\Sp(d-1,q_0){:}2|}=q^{(d-1)^2\over 8}\prod_{i=1}^{d-1\over 2}(q^i+1).
\end{align*}
Therefore, the results follow.
\end{proof}

\section{ Rotary pairs of alternating groups and symmetric groups}\label{section4}

In this section, we determine the number of inequivalent Hall rotary pairs of $\A_r$ and $\S_r$, where $r$ is a prime, by a series of lemmas.

\subsection{Alternating groups}
In this subsection, we will calculate the number of inequivalent Hall rotary pairs of $\A_r$. We will divide this subsection into three parts based on the choice of $r$.
\subsubsection{Not Singer primes}
In this part, we assume that $r$ is not a Singer prime.
\begin{lemma}\label{lem:RotofAp}
Let $G=\A_r$, where $r>5$ is a prime and $r\neq 11,23$.
Assume that $r$ is not a Singer prime.
Then the number of inequivalent Hall rotary pairs $(a,z)$ of $G$ is equal to
        ${1\over r}\inv(\A_r)-\delta(r,1)$,
      where $\delta(r,1)=1$ if $r\equiv 1\pmod 4$,  and $\delta(r,1)=0$  otherwise.
\end{lemma}

\begin{proof}
Fix $a=(1,\cdots,r)$, and set 
\[
\Delta=\{ z\in G: |z|=2, \l a,z\r=G \}.
\]
For an involution $z\in G$, set $K=\l a,z\r$, and suppose $K\neq G$.
Since $r$ is not a Singer prime, by Lemma~\ref{lem:Ar-reduction}, we have
 $K\cong \D_{2r}$ with $r\equiv 1\pmod 4$.
Since $z\in \N_G(\l a\r)$,  $K$ is the unique proper subgroup of $G$ that contains $a$.
Therefore, we conclude that
$$|\Delta|=\inv(G)-\inv(\D_{2r})=\inv(G)-\delta(r,1)r.$$
Since all elements of $G$ of order $r$ are conjugate in $\Aut(G)=\S_r$ and
$\C_{\S_r}(a)=\l a\r$ is semi-regular on $\Delta$, the number of inequivalent rotary pairs of $G$ is 
$$\frac{|\Delta| }{|\C_{\S_r}(a)|}
=\frac{\inv(G)}{r}-\delta(r,1).$$
This completes the proof.
\end{proof}

\subsubsection{Fermat primes}\label{sec:F-primes}
In this part, we assume that $r=2^{2^t}+1$ is a Fermat prime.
Let $G=\A_r$, naturally acting on $\Omega=\{1,2,\dots,r\}$.
Fix  $a=(1,2,\dots ,r)\in G$.
Then 
$$\N_{\Sym(\Ome)}(\l a\r)=\l a\r{:}\l b\r\cong\AGL(1,r)=\ZZ_r{:}\ZZ_{r-1}.$$

Suppose $q=2^{2^t}$, and let $K=\PSL(2,q){:}\ZZ_f=\PGammaL(2,q)$, where $f=2^t$.
Then $K$ acts faithfully on the set of 1-subspaces of the vector space $\FF_q^2$, which is of cardinality $q+1=r$, and thus $\PGammaL(2,q)=K<\Sym(\Omega)=\S_r$, and Singer cycles of $K$ are Hall cycles of $K$ and $\Sym(\Omega)$.
Since all Hall cycles of $G$ are conjugate, we may assume that the Hall cycle $\l a\r$ lies in $K$.
Then 
\[\mbox{$\N_K(\l a\r)=\l a\r{:}\l b^{r-1\over2f}\r=\ZZ_r{:}\ZZ_{2f}$.}\]

\begin{lemma}\label{lem:PSL(2,q)<Alt}
With the notation defined above, the following hold:
\begin{enumerate}[\rm(i)]
\item subgroups of $\S_r$ which are isomorphic to $\PSL(2,q){:}2^s\le \PGammaL(2,q)$ are all conjugate in $\S_r$;

\item the symmetric group $\Sym(\Ome)=\S_r$ has exactly $2^{2^t-t-1}$ subgroups which contain $\l a\r$ and are isomorphic to $\PSL(2,q){:}\ZZ_{2^s}$ with $2^s\le2^t$;

\item $\PSL(2,q){:}\l\phi_0\r<\A_r$, where $\phi_0=\phi^{2^{t-1}}$ is an involution;

\item letting $K_1,K_2<\Sym(\Ome)$ such that $a\in K_1\cap K_2$ and $K_1\cong K_2\cong\PSL(2,q){:}\ZZ_{2^s}$, we have that $K_1\cap K_2=\l a\r{:}\l b^{r-1\over2^{s+1}}\r=\ZZ_r{:}\ZZ_{2^{s+1}}$.
\end{enumerate}
\end{lemma}

\begin{proof}
Since subgroups of index $q+1$ in $\PGammaL(2,q)$ are the maximal parabolic subgroups of $\PGammaL(2,q)$, we conclude that they are conjugate in $\PGammaL(2,q)$. 
Hence, all subgroups of $\S_r$ which are isomorphic to $\PGammaL(2,q)$ are permutation isomorphic, and so they are conjugate in $\S_r$. 
Since $\PSL(2,q){:}2^s$ lies in unique subgroup isomorphic
to $\PGammaL(2,q)$, part (i) follows.

By Lemma \ref{lem:double-counting}, we conclude that the number of subgroups of $\S_r$ that contain $a$ and are isomorphic to $\PGammaL(2,q)$ is equal to
$${|\PGammaL(2,q)|\over |\N_{\S_r}(\PGammaL(2,q)|}\cdot {|\N_{\S_r}(\l a\r)|\over |\N_{\PGammaL(2,q)}(\l a\r)|}={|\PGammaL(2,q)|\over|\PGammaL(2,q)|}\cdot{|\AGL(1,r)|\over|\l a\r{:}\ZZ_{2f}|}={r-1\over2^{t+1}}=2^{{2^t}-t-1},$$
as stated in part~(ii).

By Lemma~\ref{lem:incl}~${\rm(3)}$, part~(iii) follows.

Finally, let $K_1,K_2<\Sym(\Ome)$ such that $\l a\r\le K_1\cap K_2$ and $K_1\cong K_2\cong\PSL(2,q){:}\ZZ_{2^s}$.
Then $\N_{K_1}(\l a\r)\cong \N_{K_2}(\l a\r)\cong \ZZ_r{:}\ZZ_{2^{s+1}}$.
As $\N_{\Sym(\Ome)}(\l a\r)=\l a\r{:}\l b\r=\ZZ_r{:}\ZZ_{r-1}$ has a unique subgroup of order $2^{s+1}r$, we conclude that 
$$\N_{K_1}(\l a\r)=\N_{K_2}(\l a\r)=\l a\r{:}\l b^{r-1\over {2^{s+1}}}\r=\ZZ_r{:}\ZZ_{2^{s+1}}.$$ 
Since $\N_{K_i}(\l a\r)$ is maximal in $K_i$ for $i=1$ and $2$,
we have $K_1\cap K_2=\l a\r{:}\l b^{r-1\over {2^{s+1}}}\r=\ZZ_r{:}\ZZ_{2^{s+1}}$, as in part~(iv).
\end{proof}

Now we can determine the Hall rotary pairs of $\A_r$ with $r$ being a Fermat prime.

\begin{lemma}\label{lem:Fermat-Ap}
Let $G=\A_r$, where $r=q+1=2^{2^t}+1$ is a Fermat prime.
Then the number of inequivalent Hall rotary pairs of $G$ is equal to 
$${1\over r}\inv(\A_r)-{2^{2^t-t-1}\over r}\inv(\PSL(2,q){:}2)-(2^{2^t-t-1}-1).$$
\end{lemma}

\begin{proof}
Fix $a=(1,\cdots,r)$, and let  $z\in G$ be an  involution such that  $K:=\l a,z\r\neq G$.
Then, by Lemma \ref{lem:Ar-reduction},  we deduce that  $K$  is isomorphic to one of the following groups:
\[\mbox{$\D_{2r}$, $\PSL(2,q)$, and $\PSL(2,q){:}\l\phi_0\r$ with $|\phi_0|=2$.}\]

It is obvious that $G$ has a unique subgroup which contains $\l a\r$ and is isomorphic to $\D_{2r}$.
Since $\N_{\PSL(2,q)}(\l a\r)=\D_{2r}$ and each $\PSL(2,q)<\PSL(2,q){:}2$, it yields that 
\[S:=\bigcup_{K<G}K=\bigcup_{K\cong \PSL(2,q){:}2}K\subset G.\]
By Lemma~\ref{lem:PSL(2,q)<Alt}~(ii), there are exactly $2^{2^t-t-1}$ subgroups which contain $\l a\r$ and are isomorphic to $\PSL(2,q){:}2$, say
\[K_1,\dots,K_n,\ \mbox{where $n=2^{2^t-t-1}$}.\]
It follows from Lemma~\ref{lem:PSL(2,q)<Alt}\,(iv) that $K_1\cap K_2\cap\dots\cap K_n=\l a\r{:}\l b^{r-1\over {4}}\r=\ZZ_r{:}\ZZ_{4}$, which contains $r$ involutions.
Thus the number of involutions in $S$ equals
\[\mbox{$\inv(S)=(\inv(K_1)-r)+\dots+(\inv(K_n)-r)+r=n\inv(\PSL(2,q){:}2)-(n-1)r$.}\]
Let
\[\Delta=\{(a,z)\mid  z\in G \mbox{ with $|z|=2$ and }\l a,z\r=G\}.\]
Then $|\Delta|=\inv(G)-\inv(S)$. 
Since all elements of $G$ of order $r$ are conjugate and $\C_\Aut(G)(\l a\r)=\l a\r=\Z_r$ is semi-regular on $\Delta$, we conclude that the number of inequivalent Hall rotary pairs of $G$ is equal to
$${1\over r}(\inv(G)-\inv(S))={1\over r}\inv(\A_r)-{2^{2^t-t-1}\over r}\inv(\PSL(2,q){:}2)-(2^{2^t-t-1}-1).$$
This completes the proof.
\end{proof}

\subsubsection{Singer primes}\label{sec:S-primes}
In this part, we assume that $r={q^d-1\over q-1}$ is a Singer prime but not a Fermat prime.
Let $G=\A_r$, acting on $\Ome=\{1,2,\dots,r\}$, and fix a Hall cycle $\l a\r$ with $a=(1,2,\dots ,r)$.

\begin{lemma}\label{lem:PSL(d,q)<Alt}
Let $r={q^d-1\over q-1}$ be a Singer prime, where $d\ge3$ is a prime and $q=p^f$ with $p$ prime.
Then the following statements hold:
\begin{enumerate}[\rm(i)]
\item $f=d^t$, where $t$ is a positive integer;

\item $\PSL(d,q){:}\ZZ_{d^s}\le \PGammaL(d,q)<\A_r$, and subgroups of $\S_r$ which are isomorphic to $\PSL(d,q){:}\ZZ_{d^s}$ are all conjugate in $\S_r$;

\item $\S_r$ has exactly ${r-1\over d^t}$ subgroups which contain $\l a\r$ and are isomorphic to $\PSL(d,q){:}\ZZ_{d^s}$;


\end{enumerate}
\end{lemma}

\begin{proof}
By Lemma~\ref{lem:Ar-reduction}~(iii),  part~(i) follows.

Note that $\PGammaL(d,q)=\PSL(d,q){:}\ZZ_{d^t}<\S_r$, and   $d^t$ is an  odd integer. Since each permutation on $\Ome$ of odd order can be written as a product of pairwise disjoint odd cycles,
it follows that $\PGammaL(d,q)<G=\A_r$.
Now $\PSL(d,q){:}\ZZ_{d^s}$ is embedded in $G$ with Hall cycles of order $r$.
Let $K<G$ be such that $K\cong\PSL(d,q){:}\ZZ_{d^s}$.
Since subgroups of index $r$ in $\PSL(d,q){:}\ZZ_{d^s}$ are conjugate in $\Aut(\PSL(d,q){:}\ZZ_{d^s})$, we conclude that all subgroups of $\S_r$ which are isomorphic to $\PSL(d,q){:}\ZZ_{d^s}$ are permutation isomorphic.
Then all subgroups of $G$ that are isomorphic to $\PSL(d,q){:}\ZZ_{d^s}$ form a unique conjugacy class of $\S_r$,
as in part~(ii).

It follows from \ref{lem:double-counting} that the number of subgroups of $\S_r$ that contain $\l a\r$ and are isomorphic to $\PGammaL(d,q)$ is equal to
$${|\PGammaL(d,q)|\over |\N_{\S_r}(\PGammaL(d,q)|}\cdot {|\N_{\S_r}(\l a\r)|\over |\N_{\PGammaL(d,q)}(\l a\r)|}
={|\PGammaL(d,q)|\over|\PGammaL(d,q)|}\cdot{|\AGL(1,r)|\over |\l a\r{:}d^{t+1}|}={r-1\over d^{t+1}},$$
as in part~(iii).
%
\end{proof}

We remark that a Singer prime may have different representations, for instance, $31={5^3-1\over5-1}={2^5-1\over2-1}$, so
\[\mbox{$\PSL(3,5),\ \PSL(5,2) < \A_{31}$.}\]

\begin{lemma}\label{lem:Singer-Ap}
Let $r$ be a Singer prime.
Then 
the number of inequivalent Hall rotary pairs of $\A_r$
equals
\[{1\over r}\inv(\A_r)-\frac{r-1}{r}\sum_{{q^d-1\over q-1}=r}\frac{\inv(\PSL(d,q))}{d^{t+1}}-\delta(r,1),\]
 where $\delta(r,1)=1$ if $r\equiv 1\pmod 4$,  and $\delta(r,1)=0$  otherwise.
\end{lemma}

\begin{proof}
Denote $G=\A_r$. Fix $a=(1,2,\cdots,r)$, and let  $z\in G$ be an  involution such that  $K:=\l a,z\r\neq G$.
Then 
\[\text{$K=\D_{2r}$ for $r\equiv1\ (\mod 4)$, or ${\PSL(d,q)}$ with $d\geq 3$}.\]
Let $S=\bigcup_{\l a,z\r<\A_r} \l a,z\r$, the union of the proper subgroups $\l a,z\r$ of $G$.

If $K\cong \D_{2r}$,
then $z\in \N_{G}(\l a\r)$,  and there is a unique subgroup that contains $a$ and is isomorphic to $\D_{2r}$.
Therefore, there are $\delta(r,1)\cdot r$ involutions  in $G$ such that $\l a,z\r=K$. 

Let  $(d_i,q_i)$, $i\in I$, be the set of solutions of the equation $\frac{q^d-1}{q-1}=r$. Then $d_i>2$ is prime, $q_i=p_i^{d_i^{t_i}}$ for some $t_i$, and $\gcd(d_i,q_i-1)=1$.
Suppose $K\cong \PSL(d_i,q_i)$. By Lemma~\ref{lem:PSL(d,q)<Alt}, the number of subgroups of $G$ that contain $a$ and are isomorphic to $\PSL(d_i,q_i)$ is ${r-1\over d_i^{t_i+1}}$.
Note that any involution $z$ in $K$  satisfies $\l a,z\r=K$ because $\N_{K}(\l a\r)\cong \l a\r{:}d_i$ is a maximal subgroup of $K$ of odd order and no other maximal subgroup contains $a$; see \cite{KL}.
Therefore, there are ${r-1\over d_i^{t_i+1}}\cdot \inv(\PSL(d_i,q_i))$ involutions in $G$ such that $\l a,z\r\cong \PSL(d_i,q_i)$.

Thus
\begin{align*}
    \inv(S)&=\sum_{\l a,z\r<\A_r} \inv(\l a,z\r)=\delta(r,1)r+\sum_{\l a,z\r\cong\PSL(d,q)}\inv(\l a,z\r)\\
    &=\delta(r,1)r+\sum_{{q^d-1\over q-1}=r}{{r-1\over d^{t+1}}}{\inv(\PSL(d,q))}.
\end{align*}
Let $$\Delta=\{(a,z)\mid z\in G\text{ is an involution and } \l a,z\r=\A_r\}.$$
Then $|\Delta|=\inv(G)-\inv(S)$. Since all elements of $G$ of order $r$ are conjugate and $\C_{\Aut(G)}(\l a\r)=\l a\r$ is semi-regular on $\Delta$, we have
there are ${|\Delta|\over r}$ inequivalent Hall rotary pairs of $\A_r$; this completes the proof.
\end{proof}

\subsection{Symmetric groups}
In this subsection, we calculate the number of Hall rotary pairs of $\S_r$, where $r$ is a prime. 
Fix a Hall cycle $\l a \r$ with $a=(1,2,\dots,r)$.
\begin{lemma}\label{lem:RotofSp}
Let  $G=\S_r$, with $r> 5$ a prime.
Then the number of  inequivalent Hall rotary pairs $(a,z)$ of  $G$ is equal to 
$${1\over r}\inv(\S_r)-{1\over r}\inv(\A_r)-\delta(r,3),$$
where $\delta(r,3)=1$ if $r \equiv 3 \pmod{4}$, and $\delta(r,3)=0$ otherwise.
\end{lemma}

\begin{proof}
Let $K=\l a,z\r$, where $z\in G$ is an involution.
If $K<\S_r$, then either $K\le\A_r$ or $K=\D_{2r}$ with $r\equiv 3\ (\mod 4)$.
Let $\Delta=\{(a,z):z\in G \textup{ with } |z|=2, \l a,z\r=G\}$. 
Then 
\[
|\Delta|
= \inv(\S_r) - \inv(\A_r) - \delta(r,3)\,r.
\]
Since all elements of $G$ of order $r$ are conjugate  and  $\C_{\Aut(G)}(\l a\r)=\l a\r$ is semi-regular on $\Delta$, the number of inequivalent Hall rotary pairs of $G$ is $\frac{|\Delta| }{r}$. 
This proves the lemma.
\end{proof}

\begin{remark}
\rm{
It is an open question to find all rotary pairs $(a,z)$ for $\A_n$ with $n\ge 5$; see Problem $1$ in \cite{ChenDuLi22}.
Lemma~\ref{lem:RotofSp} answers the question when $|a|=n$ is a prime.
}
\end{remark}

\section{Reflexible pairs of alternating groups and symmetric groups}\label{sec:refAr}
In this section, we calculate the number of inequivalent reflexible pairs of $\A_r$ and $\S_r$, where $r$ is a prime. 
We first give Construction~\ref{cons:Alt-ref}. 
Actually, for each given $a$, any involution $z$ that can form a reflexible pair with $a$ is obtained via Construction~\ref{cons:Alt-ref}, which will be proved in Lemma~\ref{lem:regofAp}.
\begin{construction}\label{cons:Alt-ref}
{\rm
In $\S_r=\Sym(\Omega)$ with $\Omega=\{1,2,\dots,r\}$, pick two elements
\[\begin{array}{l}
a=(1,2,\dots, r),\\ 
\s=(1,r)(2,r-1)\dots({r-1\over2},{r+3\over2}).
\end{array}\]
Let $z$ be an involution in $\C_{\S_r}(\s)=\S_2\wr\S_{r-1\over2}.$
Then $(a,z)^\s=(a^{-1},z)$, and $\s\in\Aut(\l a,z\r)$. That is to say, $(a,z)$ is a reflexible pair of $\l a,z\r$.
}
\end{construction}

\begin{lemma}\label{lem:KofAp}
In the notation of Construction~\ref{cons:Alt-ref}, $\l a,z\r<\A_r$ if and only if 
\begin{enumerate}[\rm(i)]
    \item $\l a,z\r=\D_{2r}$, or
    \item $r=q+1$, $\l a,z\r=\PSL(2,q)$ or $\PSL(2,q){:}2$.
\end{enumerate}
\end{lemma}
\begin{proof}
Let $K=\l a,z\r$. By Lemma~\ref{lem:Ar-reduction}, the group $K$ can be one of the following groups:
$$\{\D_{2r},\ \PSL(2,11),\ \M_{11},\ \M_{23},\ \PSL(2,q),\ \PSL(2,q){:}2,\  \PSL(d,q)\text{ with $d\ge 3$}\}.$$ 

Suppose $K=\PSL(2,11)$ and $r=11$. 
Then $(a,z)$ is a reflexible pair of $\PSL(2,11)$ and $\l a,z,\s\r=\PGL(2,11)$. 
Since $a,z,\s\in\Aut(\A_{11})$, we conclude that $\l a,z,\s\r\leq\S_{11}$.
By the Atlas~\cite{atlas} $\PGL(2,11)\not<\S_{11}$, which is a contradiction. 
So $K\neq \PSL(2,11)$. 

Suppose  $K=\M_{11}$ or $\M_{23}$. 
Then $(a,z)$ is a reflexible pair of $M_{11}$ or $M_{23}$, respectively. However, Lemmas \ref{lem:M11} and \ref{lem:M23} show that $M_{11}$ and $M_{23}$ do not have reflexible pairs. 
We thus conclude that $K$ is neither  $\M_{11}$ nor $\M_{23}$.

Suppose $K=\PSL(d,q)$ with $d \geq 3$ and $q=p^{d^t}$.
Then $(a,z)$ is a reflexible pair of $\PSL(d,q)$. 
Since $a^{\s}=a^{-1}$, $\s$ is actually a graph automorphism of $\PSL(d,q)$; see the proof of Lemma~\ref{lem:PSL-d>2-ref}. 
However, $\PSL(d,q){:}\l \sigma\r\nleq\S_r$.
Thus we get a contradiction and $K\neq \PSL(d,q)$.
\end{proof}

\begin{lemma}\label{lem:regofAp}
Let $G=\A_r$ with $r>5$ prime.
Then the following statements hold.
\begin{enumerate}[\rm(i)]
\item Each reflexible pair of $G$ is equivalent to a pair $(a,z)$ defined in Construction~$\ref{cons:Alt-ref}$.

\item The number of inequivalent reflexible pairs of $G$ is equal to 
\begin{enumerate}[{\rm(a)}]
    \item $\inv((\C_2\wr\S_{r-1\over2})\cap\A_r)-\delta(r,1),$ if $r$ is not a Fermat prime;
    \item $\inv((\C_2\wr\S_{r-1\over2})\cap \A_r)-2^{2^t-2t-2}\left(q +q^{1\over 2}-3\right)-1,$ if $r=q+1$ is a Fermat prime.
\end{enumerate}
\end{enumerate}
\end{lemma}

\begin{proof}
Since all elements of order $r$ are conjugate in $\S_r$, we may fix $a=(1,2,\dots,r)$.
Let $z\in G$ be an involution such that $(a,z)$ is a reflexible pair of $G$.
Then there exists an involution $\s\in \Aut(G)=\S_r$ such that $(a,z)^\s=(a^{-1},z)$, so $\s$ inverts $a$ and centralizes $z$.
Thus $\s\in\N_{\S_r}(\l a\r)\cap\C_{\S_r}(z)$.
In particular, $\s\in\N_{\S_r}(\l a\r)=\l a\r{:}\l b\r\cong\AGL(1,r)$.
Since $\AGL(1,r)$ has a unique conjugacy class of involutions, up to equivalence, we may take $\s=b^{r-1\over2}=(1,r)(2,r-1)\dots({r-1\over2},{r+3\over2})$.
Thus $z\in\C_{\S_r}(\s)=\S_2\wr\S_{r-1\over2}$, as in part~(i).

Assuming that $(a,z)$ is a reflexible pair of $G$, we claim that $(a,z)$ is not equivalent to any other pair defined in Construction~\ref{cons:Alt-ref}.
Suppose $(a,z')$ is a pair such that $z'\in \C_{\S_r}(\s)$ and there exists some $\alpha\in \Aut(G)$ satisfying $(a,z)^\alpha=(a,z')$.
Then $\alpha\in \C_{\S_r}( a)$, $\l a,z'\r=G$ and $z'\in \C_{\S_r}(\s^{\alpha})$. 
Since $z'\in \C_{\S_r}(\s)$, we conclude that $(z')^{\s^{-1}\s^{\alpha}}=z'$.
As $\alpha\in \C_{\S_r}( a)$ and $a^{\s}=a^{-1}$, we have $a^{\s^{-1}\s^{\alpha}}=(a^{-1})^{\alpha^{-1}\s\alpha}=a$. 
Thus $$(a,z')^{\s^{-1}\s^{\alpha}}=(a,z').$$
Since $\l a,z'\r=G$ and ${\s^{-1}\s^{\alpha}}\in \Aut(G)$, we deduce that ${\s=\s^{\alpha}}$, and so $z'=z$. Hence the claim holds.

For any involution $z\in\C_{\S_r}(\s)\cap G$, by Lemma~\ref{lem:KofAp}, $K=\l a,z\r$ can be isomorphic to one of the groups in $$\{\D_{2r},\ \PSL(2,q),\ \PSL(2,q){:}2,\ \A_r\}.$$
Hence we count reflexible pairs of $G$ according as $r$ is a Fermat prime or not.

\medskip
\noindent\textbf{Case 1: $r$ is not a Fermat prime.}

In this case, $K$ can only be $\D_{2r}$ or $\A_r$. 
Suppose $K\cong\D_{2r}$. Then by Lemma \ref{lem:Ar-reduction}~$\rm(i)$, we have $r\equiv 1\pmod 4$. 
Since $z\in \N_{\S_r}(\l a\r)\cong  \AGL(1,r)$  and $\AGL(1,r)$ contains only one subgroup isomorphic to $\D_{2r}$, there exists exactly one subgroup isomorphic to  $\D_{2r}$ that contains $a$, and thus $K=\l a,\s\r$. 
It follows that
$\inv(\C_{\S_r}(\s)\cap K)=1$, and the number of inequivalent reflexible rotary pairs of $\A_r$ is equal to $\inv(\C_{\S_r}(\s)\cap \A_r)-1$.
If $r\not\equiv 1\pmod 4$, $K=\A_r$ and the number of inequivalent reflexible rotary pairs of $\A_r$ is equal to $\inv(\C_{\S_r}(\s)\cap \A_r)$.

In conclusion, the number of inequivalent reflexible rotary pairs of $\A_r$ is equal to $$\inv(\C_{\S_r}(\s)\cap \A_r)-\delta(r,1).$$
\medskip
\noindent\textbf{Case 2: $r$ is a Fermat prime.}
 
In this case, we assume $r$ is a Fermat prime, say, $r=2^{2^t}+1$ with $t\geq 2$. 
Then $K$ is isomorphic to $\D_{2r}$, $\PSL(2,q)$, $\PSL(2,q){:}2$, or $\A_r$, where $q=2^{2^t}$. 
Since $t\ge 2$, we have $r\equiv 1\pmod 4$. 
Thus $K\cong\D_{2r}$ always occurs.
If $K\cong\D_{2r}$,  as in {\bf Case 1},
there exists a unique $\D_{2r}$ that contains $a$, and $\inv(\C_{\S_r}(\s)\cap K)=1$.

Suppose $K\cong\PSL(2,q)$. 
Lemma \ref{lem:PSL-d=2} shows that there are ${q-2\over2f}\cdot\phi(q+1)$ inequivalent reflexible pairs of $K$, where $f=2^t$.
By Lemma \ref{lem:Hall-RP}, the inequivalent reflexible pairs of $K$  have the form $(a^i,u)$ for some involution $u\in K$ and some  $i$ such that $\gcd(i, |a|)=1$.
Note that $(a,u)$ is a reflexible pair of $\PSL(2,q)$  if and only if $(a^i,u)$ is a reflexible pair of $\PSL(2,q)$.
Hence, there are ${q-2\over 2^{t+1}}$ inequivalent reflexible pairs of $K$ which have the form $(a,u)$  such that $u$ is an involution.
Suppose  $(a,u)$ is a reflexible pair of $K$ such that $(a,u)^\tau=(a^{-1},u)$ for some $\tau\in \Aut(K)\le \S_r$.  Noting that $\N_{\S_r}(\l a\r)\cong \AGL(1,r)$ contains a unique subgroup isomorphic to $\D_{2r}$, which is $\l a,\s\r$, we have  $\tau\in \l a,\s\r$.
Since $|a|=r$ is a prime,  $\tau=\s^{a^i}$ for some $0\le i\le r-1$. It then follows that $u\in \C_{\S_r}(\tau)=\C_{\S_r}(\s^{a^i})=(\C_{\S_r}(\s))^{a^i}$.  Since $(a,u)$ is equivalent to $(a,u^{a^j})$ for any $0\le j\le r-1$,
 each reflexible pair $(a,u)$ of $K$ is equivalent to a reflexible pair $(a,u')$ such that $u'\in \C_{\S_r}(\s)$.
 By Lemma~\ref{lem:PSL(2,q)<Alt}~(ii), there are $2^{2^t-t-1}$ subgroups that contain $a$ and are isomorphic to $\PSL(2,q)$. Since the sets of reflexible pairs of different subgroups isomorphic to $\PSL(2,q)$ intersect trivially, we conclude that there are $2^{2^t-t-1}\cdot{q-2\over 2^{t+1}}$  pairs $(a,z)$ which are not conjugate in $\S_r$ such that $\l a,z\r\cong\PSL(2,q)$ and $z\in\C_{\S_r}(\s)$.

Suppose that $K\cong \PSL(2,q){:}2$.
It is similar to the case where $K\cong \PSL(2,q)$.
Lemma~\ref{lem:PSL(2,q).2-ref} tells that there are ${q^{1\over 2}-1\over 2f}\cdot\phi(q+1)$ inequivalent reflexible pairs of $K$, where $f=2^t$.
Hence, the number of inequivalent reflexible pairs $(a,u)$ of $K$ such that $u\in\C_{\S_r}(\s)$ is ${q^{1\over 2}-1\over 2^{t+1}}$. By Lemma~\ref{lem:PSL(2,q)<Alt}~(ii), there are $2^{2^t-t-1}$ subgroups that contain $a$ and are isomorphic to $\PSL(2,q){:}2$.
Since the sets of reflexible pairs of different subgroups isomorphic to $\PSL(2,q){:}2$ intersect trivially,  there are $2^{2^t-t-1}\cdot{q^{1\over 2}-1\over 2^{t+1}}$ pairs $(a,z)$ which are not conjugate in $\S_r$ such that $\l a,z\r\cong\PSL(2,q){:}2$ and $z\in\C_{\S_r}(\s)$.

Finally, we conclude that the number of inequivalent reflexible rotary pairs of $\A_r$ is equal to 
\begin{align*}
    \inv(\C_{\S_r}(\s)\cap \A_r)-2^{2^t-2t-2}\left({q+q^{1\over 2}-3}\right)-1,
\end{align*}
where $q=2^{2^t}$.
This completes the proof.

\end{proof}

\def\S{{\rm S}}

The result in the next lemma was first obtained in \cite{SS24}, and a short and different proof is given below for  completeness.

\begin{lemma}\label{lem:count-Alt-1}
The number of inequivalent reflexible pairs of $\S_r$ is equal to 
$$\inv(\rmC_2\wr\S_{r-1\over2})-\inv((\rmC_2\wr\S_{r-1\over2})\cap \A_r)-\delta(r,3).$$ 
\end{lemma}

\begin{proof}
As in  the proof of Lemma~\ref{lem:regofAp}, 
each reflexible pair of $\S_r$ is equivalent to a pair $(a,z)$ such that $z\in \C_{\S_r}(\s)$, where $a,\s$ are defined in Construction~\ref{cons:Alt-ref}. 
Suppose that $K=\l a,z \r$ and $K<G$. 
Then either $K=\D_{2r}$ and $r\equiv3(\mod 4)$ or $K\leq \A_r$.
Therefore, $\l a,z\r=\S_r$ if and only if $z\in\C_{\S_r}(\s)\setminus\C_{\A_r}(\s)$ and $\l a,z\r\neq \D_{2r}$.

By the proof of Lemma~\ref{lem:regofAp}, we know that $(a,z)$ is not equivalent to the other pairs defined in Construction~\ref{cons:Alt-ref}.
Hence the number of inequivalent reflexible pairs of $\S_r$ is
$$\inv(\C_{\S_r}(\s)\setminus\C_{\A_r}(\s))-\delta(r,3)=\inv(\C_{\S_r}(\s))-\inv(\C_{\A_r}(\s))-\delta(r,3).$$
Since $\C_{\S_r}(\s)=\rmC_2\wr\S_{r-1\over 2}$, we have $\C_{\A_r}(\s)=(\rmC_2\wr\S_{r-1\over 2})\cap \A_r$, and this completes the proof.
\end{proof}

\def\RegMap{{\rm RegMap}}


\section{Main theorems and their proofs}\label{sec:number}
In this section, we first give formulas for $\inv(\cdot)$ that appeared in this paper, and then we prove Theorems~\ref{thm:rotG}, \ref{thm:rot}, and Corollary~\ref{cor:orientation}.

\begin{lemma}\label{prop:PSL}
Let $d>2$ be a prime and  $q=p^f$ be such that $\gcd(d,q-1)=1$, where $p$ is a prime.
Then the following statements hold.
\begin{enumerate}[\rm(i)]
    \item
    Let $\binom{d}{i}_q=\frac{(q^d-1)\cdots(q^{d-i+1}-1)}{(q^i-1)\cdots(q-1)}$ be the Gaussian binomial coefficient. Then
    \[
    \inv(\PSL(d,q))=
    \begin{cases}
        \sum\limits_{i=1}^{\frac{d-1}{2}} q^{i(d-i)}\binom{d}{i}_q, &\textup{ if } p>2,\\
   \prod\limits_{i=1}^{d-1 \over 2} q^{i(i-1)\over2}\cdot \frac{\prod\limits_{j=i+1}^d(q^j-1)}{\prod\limits_{j=1}^{d-2i}(q^j-1)}, \, & \textup{ if } p=2.
    \end{cases}
    \]
   \item Assume  $q$ is odd. Then $$\inv(\PSO(d,q))=\begin{cases}
   \inv(\PGL(2,q))=q^2, &\textup{ if }
      d=3,\\
      \inv(\PGSp(4,q))=q^4(q^2+2), &\textup{ if }
      d=5,\\
      \sum\limits_{k=1}^{d-1\over 2}q^{k(d+1-2k)}\frac{\prod\limits_{i={d+1-2k\over 2}}^{d-1\over 2}(q^{2i}-1)}{(q^{2k}-1)\prod\limits_{j=1}^{k-1}(q^{2j}-1)}, &\textup { if } d\ge 7.
   \end{cases}$$
   \item Assume $q$ is even. Then 
   $$\inv(\Sp(d-1,q))
=\sum_{\substack{s\text{ even}\\ 2\le s\le \frac{d-1}{2}}}
   q^{\frac{s(s+2)}{4}}
   \frac{\prod\limits_{i=\frac{d+1}{2}-s}^{\frac{d-1}{2}}(q^{2i}-1)}
        {(q^{s}-1)\prod\limits_{i=1}^{\frac{s}{2}}(q^{2i}-1)}
   + \sum_{\substack{s\text{ odd}\\ 1\le s\le \frac{d-1}{2}}}
   q^{\frac{s^{2}-1}{4}}
   \frac{\prod\limits_{i=\frac{d+1}{2}-s}^{\frac{d-1}{2}}(q^{2i}-1)}
        {\prod\limits_{i=1}^{\frac{s-1}{2}}(q^{2i}-1)}.$$
\end{enumerate}

\end{lemma}
\begin{proof}

We first calculate $\inv(\PSL(d,q))$.
Since $\gcd(d,q-1)=1$, we have $\PSL(d,q)=\PGL(d,q)$.
Suppose that $p\neq 2$.
By \cite[Table B.1]{Burness16}, there are $\frac{d-1}{2}$ conjugacy classes of involutions in $\PGL(d,q)$.
For each $i=1,\dots,\frac{d-1}{2}$, there is a class of involutions whose   centralizers  have order $\frac{1}{q-1}|\GL(i,q)|\cdot|\GL(d-i,q)|$, hence the size of this class   is
\[
\frac{|\PSL(d,q)|}{\frac{1}{q-1}|\GL(i,q)|\cdot|\GL(d-i,q)|}
= \frac{(q-1)|\PSL(d,q)|}{|\GL(i,q)|\cdot|\GL(d-i,q)|}.
\]
Summing over all classes gives
\begin{align*}
\inv(\PSL(d,q))
&= \sum_{i=1}^{\frac{d-1}{2}} \frac{(q-1)|\PSL(d,q)|}{|\GL(i,q)|\cdot|\GL(d-i,q)|} \\
&= \sum_{i=1}^{\frac{d-1}{2}} \frac{q^{i(d-i)}(q^d-1)\cdots(q^{d-i+1}-1)}{(q^i-1)\cdots(q-1)} \\
&= \sum_{i=1}^{\frac{d-1}{2}} q^{i(d-i)}\binom{d}{i}_q.
\end{align*}
Suppose $p=2$.
Table B.3 of \cite{Burness16} shows that $\Aut(\PSL(d,q))$ has $\frac{d-1}{2}$ involution classes.
By \cite[Proposition 3.2.7]{Burness16}, $x^{\PSL(d,q)}=x^{\Aut(\PSL(d,q))}$ for any involution $x\in\PSL(d,q)$ because $(d,q-1)=1$.
From the first row of that table, for each $1\le i\le\frac{d-1}{2}$ there is a class of involutions in $\PSL(d,q)$ of size
\begin{align*}
N_i
&= \frac{(q-1)\cdot|\PSL(d,q)|}{q^{i(2d-3i)}\cdot|\GL(i,q)|\cdot|\GL(d-2i,q)|} \\
&= q^{\frac{i(i-1)}{2}}\,
   \frac{\prod_{j=i+1}^{d}(q^{j}-1)}{\prod_{j=1}^{d-2i}(q^{j}-1)} .
\end{align*}
Hence
$\inv(\PSL(d,q)) = \sum_{i=1}^{\frac{d-1}{2}} N_i$.
Now $\rm(i)$ follows.

Next, we calculate $\inv(\PSO(d,q))$ with $q$ odd.
Note that $\PSO(3,q)\cong \PGL(2,q)$ and $\PSO(5,q)\cong \PGSp(4,q)$.
By \cite[Table B.1]{Burness16}, there are two conjugacy classes of involutions in $\PGL(2,q)$ and
\[
\inv(\PGL(2,q))=\frac{|\PGL(2,q)|}{2(q-1)}+\frac{|\PGL(2,q)|}{2(q+1)}=q^2.
\]
By \cite[Table B.5]{Burness16}, there are $4$ conjugacy classes of involutions in $\PGSp(4,q)$ and
\begin{align*}
\inv(\PGSp(4,q))
&= \frac{|\PGSp(4,q)|}{2|\Sp(2,q)|^2} + \frac{|\PGSp(4,q)|}{2|\Sp(2,q^2)|}
   + \frac{|\PGSp(4,q)|}{2|\GL(2,q)|} + \frac{|\PGSp(4,q)|}{2|\GU(2,q)|} \\[2pt]
&= \frac{q^2(q^4-1)}{2(q^2-1)} + \frac{q^2(q^2-1)}{2}
   + \frac{q^3(q^4-1)}{2(q-1)} + \frac{q^3(q^4-1)}{2(q+1)} \\[2pt]
&= q^4(q^2+2).
\end{align*}
When $d\ge 7$, we have $\PSO(d,q)={\rm PGO}(d,q)$.
By \cite[Table B.8]{Burness16}, there are $d-1$ conjugacy classes of involutions in ${\rm PGO}(d,q)$. 
Write $d=2m+1$.  Then
{\footnotesize\begin{align*}
\inv({\rm PGO}(d,q))
&= \sum_{1\le k\le m}
   \Biggl(
      \frac{|{\rm PGO}(d,q)|}{2|\SO^{+}(2k,q)|\cdot|\SO(d-2k,q)|}
      + \frac{|{\rm PGO}(d,q)|}{2|\SO^{-}(2k,q)|\cdot|\SO(d-2k,q)|}
   \Biggr) \\[2pt]
&= \frac{1}{2} \sum_{1\le k\le m}
   \frac{q^{m^2}\prod_{i=1}^{m}(q^{2i}-1)}
        {q^{k(k-1)}\prod_{i=1}^{k-1}(q^{2i}-1)
         \cdot q^{(m-k)^2}\prod_{i=1}^{m-k}(q^{2i}-1)}
   \left(\frac{1}{q^k-1}+\frac{1}{q^k+1}\right) \\[2pt]
&= \sum_{1\le k\le m}
   q^{k(d+1-2k)}
   \frac{\prod_{i=\frac{d+1-2k}{2}}^{\frac{d-1}{2}}(q^{2i}-1)}
        {(q^{2k}-1)\prod_{j=1}^{k-1}(q^{2j}-1)} .
\end{align*}}
Therefore, $\rm(ii)$ follows.

Finally, we count $\inv(\Sp(d-1,q))$ with $q$ even.
When $d=3$, since $\Sp(2,q)\cong \SL(2,q)\cong \PSL(2,q)$, 
$\inv(\Sp(2,q))=\inv(\PSL(2,q))=q^2-1$, which satisfies $\rm (iii)$.
Now assume $d>3$. Note that $\Sp(d-1,q)\cong \PSp(d-1,q)\cong \PGSp(d-1,q)$ since $q$ is even. 
By \cite[Table B.7]{Burness16}, writing $m=\frac{d-1}{2}$, we have
\begin{align*}
\inv(\PGSp(d-1,q))
&=   \sum_{\substack{s\text{ even}\\ 2\le s\le\frac{d-1}{2}}}
      \Bigg(
         \frac{|\PGSp(d-1,q)|}{q^{(d-1)s-\frac{3s^2}{2}+\frac{s}{2}}\,|\Sp(s,q)|\cdot|\Sp(d-1-2s,q)|} \\
      &\qquad\qquad
         + \frac{|\PGSp(d-1,q)|}{q^{(d-1)s-\frac{3s^2}{2}+\frac{3s}{2}-1}\,|\Sp(s-2,q)|\cdot|\Sp(d-1-2s,q)|}\Bigg) 
      \\
      &\quad
      + \sum_{\substack{s\text{ odd}\\ 1\le s\le\frac{d-1}{2}}}
      \frac{|\PGSp(d-1,q)|}{q^{(d-1)s-\frac{3s^2}{2}+\frac{s}{2}}\,|\Sp(s-1,q)|\cdot|\Sp(d-1-2s,q)|}
   \\[2mm]
   &= \sum_{\substack{s\text{ even}\\ 2\le s\le m}}
   q^{\frac{s(s+2)}{4}}
   \frac{\prod\limits_{i=m-s+1}^{m}(q^{2i}-1)}
        {(q^{s}-1)\prod\limits_{i=1}^{\frac{s}{2}}(q^{2i}-1)}
   + \sum_{\substack{s\text{ odd}\\ 1\le s\le m}}
   q^{\frac{s^{2}-1}{4}}
   \frac{\prod\limits_{i=m-s+1}^{m}(q^{2i}-1)}
        {\prod\limits_{i=1}^{\frac{s-1}{2}}(q^{2i}-1)}.
\end{align*}
Now, $\rm(iii)$ follows.
\end{proof}

\begin{lemma}\label{lem:Ap}
Let $n\ge 5$ be an integer. Then the following statements hold:
\begin{itemize}
    \item[\rm(i)] $\inv(\S_n)=\sum\limits_{k=1}^{\lfloor{n\over 2}\rfloor}\frac{n!}{(n-2k)!\cdot2^k\cdot k!}$;
    \item[\rm(ii)] $\inv(\A_n)=\sum\limits_{\substack{k\text{ even}\\1\le k\le \lfloor{n\over 2}\rfloor}}\frac{n!}{(n-2k)!\cdot2^k\cdot k!}$.
\end{itemize}
\end{lemma}
\begin{proof}
Note that each involution in $\S_n$ is a product of $k$ disjoint transpositions for some $k$ with $1\le k\le \lfloor\frac{n}{2}\rfloor$.
Let $I_k$ be the set of involutions which are products of $k$ disjoint transpositions.
Then $I_k$ is a conjugacy class in $\S_n$, and the involution $z=(12)\cdots(2k-1,2k)$ belongs to $I_k$.
The centralizer
\[\C_{\S_n}(z)=\C_{\Sym\{1,\dots,2k\}}(z)\times\Sym\{2k+1,\dots,n\}=(\rmC_2\wr\S_k)\times\S_{n-2k}.\]
Thus $|I_k|={n!\over 2^k\cdot k!\cdot (n-2k)!}$, and the number of involutions of $\S_n$ equals
\[|I_1|+|I_2|+\dots+|I_{[{n\over2}]}|=\sum_{k=1}^{\lfloor{n\over2}\rfloor}\frac{n!}{(n-2k)!\cdot2^k\cdot k!}.\]
And this proves $\rm(i)$.

Each involution of $\A_n$ is a product of an even number of  transpositions.
Therefore, involutions in $I_k$ belong to $\A_n$ if and only if $k$ is even. So $\rm(ii)$ follows.
 \end{proof}

\begin{lemma}\label{lem:invA}
  Let $n = 2m+1$ be an odd integer. Then 
   \begin{enumerate}[\rm (i)]
      
       \item  $\inv(\rmC_2\wr \S_m)=\sum\limits_{{\substack{ 1\le i\le {m} }}}\binom{m}{i}\Big(\sum\limits_{\substack{ k \textup{ even}\\0\le k\le i} }\binom{i}{k}\big((k-1)!!\big)2^{{k\over2}}\Big);$

 \item $\inv((\rmC_2\wr\S_{m})\cap \A_n)=\sum\limits_{{\substack{ i \textup{ even}\\1\le i\le {m} }}}\binom{m}{i}\Big(\sum\limits_{\substack{ k \textup{ even}\\0\le k\le i} }\binom{i}{k}\big((k-1)!!\big)2^{{k\over2}}\Big).$
       
   \end{enumerate}

\
\end{lemma}
\begin{proof} 
Take $\sigma = (2,n)(3,n-1)\cdots(m+1,m+2) \in \S_n$. 
Then $$\C_{\S_n}(\sigma)\cong \rmC_2\wr \S_m,\, \C_{\A_n}(\sigma)\cong(\rmC_2\wr \S_m)\cap \A_n.$$
Set $\mathcal{B}= \bigl\{\{2,n\},\{3,n-1\},\dots,\{m+1,m+2\}\bigr\}$.
Any involution $\tau \in \C_{\S_n}(\sigma)$ stabilizes $\mathcal{B}$ and is a product of some disjoint transpositions. 
Suppose that $\tau$ moves $2i$ points on the set $\{2,\cdots,n\}$ of size $2m$, where $1\le i\le  m$.
Since $\tau$ stabilizes $\mathcal{B}$, there exist $i$ blocks in $\mathcal{B}$ such that the union of those $i$ blocks is exactly the set of points moved by $\tau$. 
Let $k$ be the number of blocks that are swapped in pairs; necessarily $k$ is even and $0\le k\le i$. 
The remaining $i-k$ blocks are fixed setwise but flipped internally.
Without loss of generality, we can suppose $\mathcal{B}'=\{\{2,n\},\{3,n-1\},\cdots, \{k+1,n+1-k\}\}$ is the set of  blocks which are moved by $\tau$. Note that   the number of possible pairs of blocks in $\mathcal{B}'$ that are swapped by $\tau$ is $(k-1)!!$, and for each such pair, there are two choices of $\tau$ acting on it. Thus there are $(k-1)!! 2^{k/2}$ choices for the action of $\tau$ on $\mathcal{B}'$.
So for each $i$, the number of choices of $\tau$ is  $$\sum_{\substack{ k \textup{ even}\\0\le k\le i}}\binom{i}{k}\big((k-1)!!\big)2^{k\over 2},$$
with the convention $\binom{i}{0}(-1)!!=1$.
It follows that  the number of involutions in $\C_{\S_n}(\sigma)$ is 
$$\sum_{1\le i\le m }\binom{{n-1\over 2}}{i}\Big(\sum_{\substack{ k \textup{ even}\\0\le k\le i} }\binom{i}{k}\big((k-1)!!\big)2^{{k\over2}}\Big).$$
For involutions $\tau$ in $\C_{\A_n}(\sigma)$, the only restriction is that $\tau$ is  an even permutation, so $i$ must be even.
It thus follows that the number of involutions in $\C_{\A_n}(\sigma)$ is 
$$\sum_{{\substack{ i \textup{ even}\\1\le i\le m }}}\binom{{n-1\over 2}}{i}\Big(\sum_{\substack{ k \textup{ even}\\0\le k\le i} }\binom{i}{k}\big((k-1)!!\big)2^{{k\over2}}\Big).$$
This completes the proof.
\end{proof}

\begin{proof}[\bf Proof of Theorem \ref{thm:rotG}:]
    Note that  $G$  is one of the groups listed in Table \ref{tab:1}.  
For part (i), the results follow from Section \ref{section3}.

For part (ii), 
if $\soc(G)$ is a linear group, by Lemmas~\ref{lem:PSL-d=2}, \ref{lem:PSL(2,q).2}, \ref{lem:PSL-d>2}, \ref{4.4}, \ref{lem:PSL(2,q).2-ref}, \ref{lem:PSL-d>2-ref} and \ref{lem:PSL.2-d>2-ref}, the results follow.

If $\soc(G)$ is an alternating group, by Lemmas~\ref{lem:RotofAp}, \ref{lem:Fermat-Ap}, \ref{lem:Singer-Ap}, \ref{lem:RotofSp}, \ref{lem:regofAp} and \ref{lem:count-Alt-1},  we conclude that $\rot(G)$ and $\Reg(G)$ are as in Table~\ref{tab:rot} and \ref{tab:reg}, respectively.
\end{proof}

\begin{proof}[\bf Proof of Theorem~\ref{thm:rot}:]
Firstly,  suppose that $G$ is solvable. 
Since $G=H\l a\r=\l a,z\r$ such that $\gcd(|H|,|a|)=1$ and $H$ is core-free in $G$,
we can directly check that $\l a\r$ is the Fitting subgroup of $G$ (see Lemma $2.3$ in \cite{DiGuoLi}).
Hence 
$$G=\l a\r{:}H=\l a\r{:}\l z\r\cong \Z_{|a|}{:}\Z_2,$$
where $|a|$ is odd and $H\cong \l z\r\cong\Z_2$.
Then each involution in $G$, together with $a^j$,  generates $G$, where $\gcd(j,|a|)=1$.
So the number of Hall rotary pairs of $G$ is $\varphi(|a|)\cdot \inv(G)$, where $\varphi$ is the Euler totient function.
Since $\Aut(G)$ is semiregular on the set of rotary pairs, we have 
$$\rot(G)=\frac{\varphi(|a|)\cdot \inv(G)}{|\Aut(G)|}.$$
As $\gcd(|a|,2)=1$, by Theorem~$8.4.2$ in \cite{kur}, we conclude that
$$\l a\r=\mathbf{C}_{\l a\r}(H)\times [\l a\r,H]\cong \Z_{k_1}\times \Z_{k_2}.$$ 
So $G\cong\Z_{k_1}\times \D_{2k_2}$ and $\gcd(k_1,k_2)=1$.
Thus $\Aut(G)=\Aut(\ZZ_{k_1})\times \Aut(\D_{2k_2})$ and $\inv(G)=k_2$.
Since $|\Aut(\ZZ_{k_1})|=\varphi(k_1)$ and $|\Aut(\D_{2k_2})|=\varphi(k_2)k_2$, we conclude that 
$$\rot(G)=\frac{\varphi(|a|)\cdot \inv(G)}{|\Aut(G)|}=\frac{\varphi(k_1k_2)\cdot k_2}{\varphi(k_1)\cdot \varphi(k_2)k_2}=1.$$

Observe that $\sigma:a\mapsto a^{-1}, z\mapsto z$ is an automorphism of $G$. 
Thus each rotary pair is reflexible, and $\Reg(G)=\rot(G)=1$.
This proves (i).


Now we deal with the case where $G$ is nonsolvable.
Since $G=H\l a\r$ is a Hall factorization and $H$ is core-free in $G$, by Theorem $1.3$ in \cite{DiGuoLi},  we have that 
\begin{align*}
    G&=\big(\l a_0\r\times T_1\times\dots\times T_n\big){.}\l c_0c_1\dots c_n\r,\\
    \l a\r&=\l a_0\r\times\l a_1\r\times\cdots\times\l a_n\r,
\end{align*}
where $a_i\in T_i$, $\l c_0c_1\cdots c_n\r\leq\Aut(\l a_0\r)\times\Out(T_1)\times\cdots\times \Out(T_n)$ such that $\l c_i\r\le \Out(T_i)$ for $1\le i\le n$ and $\l c_0\r\le \Aut(\l a_0\r)$.
Since $G=\l a,z\r$ and $a\in \l a_0\r\times T_1\times\dots\times T_n$, by taking the quotient of $G$ by $\l a_0\r\times T_1\times\cdots\times T_n$, we have $|c_0c_1\cdots c_n|\leq |z|= 2$ and $|c_i|\le 2$.
Let $G_i=T_i{.}\l c_i\r$ for $1\le i\le n$. It follows directly that $G\leq (\l a_0\r{:}\l c_0\r)\times G_1\times\dots\times G_n$.
Since $G=\l a,z\r$ and $a_i,z_i$ are the projections of $a,z$ on $G_i$, respectively, we have  $G_i=\l a_i,z_i\r$ and  $(a_i,z_i)$ is a Hall rotary pair of $G_i$.
Let $G_i=H_i\l a_i\r$. Then 
 $(G_i,H_i,a_i)$ is a Hall triple in Table~$\ref{Tab:AS-candidates}$ with $T_i=\soc(G_i)$, $T_i\ncong T_j$ for $1\le i\neq j\le n$.
 Since $G_i=H_i\l a_i\r$ is a Hall factorization and $T_i\lhd G_i$, by Lemma~2.1 in \cite{DiGuoLi}, we have $T_i=(H_i\cap T_i)\l a_i\r$ as $a_i\in T_i$.
 Similarly, since $G=H\l a\r$ is a Hall factorization and $T_i\lhd G$, it follows that $T_i=(H\cap T_i)\l a_i\r$ is a Hall factorization.
Note that for a given simple group in Table ~$\ref{Tab:AS-candidates}$, the Hall factorization is unique up to isomorphism.
 Thus  we have $H_i\cap T_i=H\cap T_i\leq H$.
Noting $(H_1\cap T_1)\times\cdots\times (H_n\cap T_n)\lhd H$,
 by taking the quotient of $H$ by $(H_1\cap T_1)\times\cdots\times (H_n\cap T_n)$, we conclude that 
 \begin{align*}
     H/((H_1\cap T_1)\times\dots\times (H_n\cap T_n))&\cong H(T_1\times\cdots\times T_n)/(T_1\times\cdots\times T_n)\\&\cong G/(\l a_0\r\times T_1\times\dots\times T_n)\\
    & \cong \l c_0c_1\cdots c_n\r.
 \end{align*}
 Thus
 $H=((H_1\cap T_1)\times\dots\times (H_n\cap T_n)){.}\l c_0c_1\dots c_n\r$.

We then  calculate $\rot(G)$ and $\Reg(G)$. 
Before that,  set $(a',z')=(a_0'a_1'\cdots a_n',z_0'z_1'\cdots z_n')$, and  we first claim that $(a',z')$ is a rotary (reflexible) pair of $G$ if and only if $(a_i',z_i')$ is a rotary (reflexible) pair of $G_i$ for $1\leq i\leq n$ and $\l a_0',z_0'\r=\l a_0\r{:}\l c_0\r$.
The necessary condition is showed above. Now we prove the sufficient condition. 
Firstly, suppose that $(a_i',z_i')$ is a rotary pair of $G_i$ for $1\leq i\leq n$ and $\l a_0',z_0'\r=\l a_0\r{:}\l c_0\r$. 
Then $\l a',z'\r\leq (\l a_0\r{:}\l c_0\r)\times G_1\times\dots\times G_n$.
We only need to show that $\l a',z'\r=G$.
Since  $\l a_i',z_i'\r=G_i$ and $T_i\ncong T_j$, by Goursat's lemma, we have $T_i\le \l a',z'\r$ for any $1\le i\le n$.
So $$\l a_0\r\times T_1\times\cdots\times T_n\leq \l a',z'\r.$$
If $|c_0c_1\cdots c_n|=1$, $\l a',z'\r=G$ follows.
Suppose that $|c_0c_1\cdots c_n|=2$.
Since $a'\in \l a_0\r\times T_1\times\cdots\times T_n$ and $z'\notin \l a_0\r\times T_1\times\cdots\times T_n$, we have 
$$|\l a',z'\r|=2\cdot|\l a_0\r\times T_1\times\cdots\times T_n|=|G|.$$
Thus $(a',z')$ is a rotary pair of $G$.
Secondly, suppose that $(a_i',z_i')$ is a reflexible pair of $G_i$ for $1\leq i\leq n$. 
Then there exists $\s_i\in \Aut(G_i)$ such that $(a_i',z_i')^{\s_i}=(a_i'^{-1},z_i')$.
In the proof of part ${\rm(i)}$, we know that there exists $\s_0\in \Aut(\l a_0\r{:}\l z_0\r)$ such that $(a_0',z_0')^{\s_0}=((a_0')^{-1},z_0')$.
Observe that $\s=\s_0\s_1\cdots\s_n$ is an automorphism of $G$ such that $(a',z')^\s=(a'^{-1},z')$. 
Hence $(a',z')$ is a reflexible pair of $G$ and the claim holds.

Now, we prove the formulas for  $\rot(G)$ and $\Reg(G)$ according as $|c_0|=2$ or $1$, respectively.
Assume that $|c_0|=2$. Let $m_0$ be the number of rotary pairs of $\l a_0\r{:}\l c_0\r$, and let $m_i$ be the number of rotary pairs of $G_i$ for $1\leq i\leq n$.
By the claim, we conclude that the number of rotary pairs of $G$ is $m_0m_1\dots m_n$.
Since $\Aut(G)=\Aut(\l a_0\r{:}\l c_0\r)\times\Aut(T_1)\times\cdots\times\Aut(T_n)$ is semiregular on the set of rotary pairs of $G$,
we have 
$$\rot(G)={m_0m_1\dots m_n\over|\Aut(\l a_0\r{:}\l c_0\r)|\cdot|\Aut(T_1)|\cdots|\Aut(T_n)|}=\rot(\l a_0\r{:}\l c_0\r)\cdot\prod_{i=1}^n \rot(G_i).$$
By part~$\rm(i)$, $\rot(\l a_0\r{:}\l c_0\r)=1$. 
So $\rot(G)=\prod_{i=1}^n \rot(G_i).$
Similarly,  we can directly check that
$$\Reg(G)=\Reg(\l a_0\r{:}\l c_0\r)\cdot \prod_{i=1}^n \Reg(T_i{:}\l c_i\r)= \prod_{i=1}^n \Reg(T_i{:}\l c_i\r).$$

Suppose $|c_0|=1$. Then $\l a_0\r{:}\l c_0\r=\l a_0\r$ and $z_0=1$.
Let $m_i$ be the number of rotary pairs of $G_i$ for $1\leq i\leq n$.
By the claim, we conclude that the number of rotary pairs of $G$ is $\varphi(|a_0|)m_1\dots m_n$.
Since $\Aut(G)=\Aut(\l a_0\r)\times\Aut(T_1)\times\cdots\times\Aut(T_n)$ is semiregular on the set of rotary pairs of $G$,
we have 
$$\rot(G)={\varphi(|a_0|)m_1\dots m_n\over|\Aut(\l a_0\r)|\cdot|\Aut(T_1)|\cdots|\Aut(T_n)|}=\frac{\varphi(|a_0|)}{|\Aut(\l a_0\r)|}\cdot \prod_{i=1}^n \rot(G_i).$$
Thus $\rot(G)= \prod_{i=1}^n \rot(G_i)$ as $|\Aut(\l a_0\r)|=\varphi(|a_0|)$.
Since there exists $\s_0\in \Aut(\l a_0\r)$ such that $a_0^{\s_0}=a_0^{-1}$ and $\Aut(G)$ is semiregular on the set of reflexible pairs of $G$, by the claim, we conclude that 
$$\Reg(G)=\frac{\varphi(|a_0|)}{|\Aut(\l a_0\r)|}\cdot \prod_{i=1}^n \Reg(T_i{:}\l c_i\r)= \prod_{i=1}^n \Reg(T_i{:}\l c_i\r).$$
Since $G_i$ is an almost simple group listed in Table~$\ref{Tab:AS-candidates}$, the values of $\rot(G_i)$ and $\Reg(G_i)$ are given in Table~$\ref{tab:rot}$ and Table~$\ref{tab:reg}$, respectively. 
So {\rm(ii)} holds.
\end{proof}

\begin{proof}[\bf Proof of Corollary~\ref{cor:orientation}:]
Let $\calM=\CayM(G,H,a,z)$ be a map satisfying Corollary~\ref{cor:orientation}. Then $H$ is a Hall subgroup of $G$, and $G,H,a,z$ are described in Theorem~\ref{thm:rot}.

Since $\calM$ is orientable if $\calM$ is rotary, we only need to consider the case where $\calM$ is bi-rotary. 
By Lemma~$3.1$ in \cite{Birotary}, when $\calM$ is bi-rotary, $\calM$ is orientable if and only if $\l a,zz^a\r$ is a subgroup of $G$ of index $2$.

Suppose that $G=\l a\r{:}\l z\r$. 
We have $zz^a=(a^{-1})^za\in \l a\r$. 
Thus $\l a,zz^a\r$ is a subgroup of $G$ of  index $2$, and $\calM$ is orientable.

Assume that $G=(\l a_0\r\times T_1\times\dots\times T_n){:}\l c_0 c_1\dots c_n\r$. 
If $|c_0c_1\dots c_n|=1$, then $G$ has no subgroup of index $2$ that contains $\l a\r$, and
$\calM$ is non-orientable.
On the other hand, assume that $|c_0c_1\dots c_n|=2$.
Then
$$zz^a=[z,a]\in G'\leq \l a_0\r\times T_1\times\cdots\times T_n.$$
Hence $\l a,zz^a\r\leq \l a_0\r\times T_1\times\cdots\times T_n$. 
Since $\l a,zz^a\r$ has index $1$ or $2$ in $G=\l a,z\r$ by~\cite{Birotary}, we have $\l a,zz^a\r= \l a_0\r\times T_1\times\cdots\times T_n$.
Therefore, $\calM$ is orientable.
\end{proof}

\section{Hall Cayley maps}\label{sec:4}
In this section, we study $G$-vertex-rotary Hall Cayley maps of $H$ which are not core-free.
Recall that $\CayM(G,H,a,z)$ denotes the maps that are $G$-vertex-rotary Cayley maps of $H$ defined by a rotary pair $(a,z)$. 
\begin{lemma}\label{quotient-graph}
Let $\mathcal{M}=\CayM(G,H,a,z)$, where  $H$ is a Hall subgroup of $G$.
Then either
\begin{enumerate}[\rm(i)]
    \item $\mathcal{M}$ is balanced; or
    \item $\mathcal{M}$ is a cover of $\cayM(\ov{G},\ov{H},\ov{a},\ov{z})$ where $(\ov G,\ov H,\ov a,\ov z)$ is a tuple $(G,H,a,z)$ described in Theorem~\ref{thm:rot}.
\end{enumerate}
\end{lemma}
\begin{proof}
   Let $L=\core_{G}(H)$. If $L=H$, i.e., $H\lhd G$, we conclude that $\mathcal{M}$ is balanced and {\rm(i)} holds.

   Now suppose that  $L<H$. Then $H\ntrianglelefteq G$.
   Let $\ov G=G/L$, $\ov H=H/L$, $\ov a=aL$, and $\ov z=zL$; and let $\ov {v}=v^L$, $\ov {e}=e^L$, and $\ov {f}=f^L$ be the orbits of $v$, $e$, and $f$, respectively, where $(v,e,f)$ is a flag of $\mathcal{M}$.
   
   Recall that $\mathcal{M}_L$ is the quotient map of $\calM$ for $L$.
   We claim that $\ov G$ acts faithfully on $\mathcal{M}_L$ and thus $\ov G\le \Aut\mathcal{M}_L$. 
   Suppose, to the contrary,  $\ov G_{\ov {v}}\cap \ov G_{\ov {e}}\cap \ov G_{\ov {f}}\neq 1$.
   If $z\in L$, then $\l a,z\r\le L\l a\r <H\l a\r=G$, which is a contradiction. 
   Thus $z\notin L$ and $\ov G_{\ov {e}}=\l \ov z\r\cong \Z_2$. 
   It follows that $\ov G_{\ov {v}}\cap \ov G_{\ov {e}}\cap \ov G_{\ov {f}}=\l {\ov z}\r,$ which implies that $\ov {z}\in \l\ov {a}\r$. 
   Then  $\ov {G}=\l\ov {a}\r$, and so $G=L\l a\r$. 
   Since $G=H\l a\r$ with $H\cap \l a\r=1$ and $L\le H$, it must be $H=L$, a contradiction.
   So the claim holds.

   Next, we show that $\mathcal{M}_L=\calM(\ov G,\ov H,\ov a,\ov z)$. 
   Note that $(\ov {a},\ov {z})$ is a rotary pair of $\ov {G}$, $\ov {G}_{\ov {v}} =\l\ov {a}\r$ and $\ov {G}_{\ov {e}}=\l\ov {z}\r$. 
   By \cite[Lemma~3.1]{LPS}, $\mathcal{M}_L$ is a $\ov {G}$-vertex-rotary map with underlying graph $\Cos(\ov {G},\l\ov {a}\r,\l\ov {z}\r)$.
   Since $H$ is regular on the vertex set of $\calM$, we have that ${\ov H=H/L}$ acts transitively on the vertex set of $\mathcal{M}_L$ and $H\cap G_v=1$.
   It follows that $\ov {H}$ is regular on the vertex set of $\mathcal{M}_L$ since $L\leq H$. 
   Thus $\mathcal{M}_L$ is a Cayley map.
   Since $\gcd(|\ov {H}|,|\ov {a}|)=1$, $\mathcal{M}_L=\calM(\ov G,\ov H,\ov a,\ov z)$. 
   Therefore, {\rm(ii)} holds.
\end{proof}

In view of Lemma \ref{quotient-graph}, we distinguish between the balanced and  non-balanced cases.

\subsection{Balanced case}
In this case, we assume that $\mathcal{M}=\cayM(G,H,a,z)$ is a balanced map such that $H$ is a Hall subgroup of $G$. Then $G=H{:}\l a\r$ and $\gcd(|H|,|a|)=1$.
Consequently, classifying $\cayM(G,H,a,z)$ is equivalent to solving the following problem.
\begin{problem}\label{prob:balance}
    Classify the triples $(H,a,z)$ where $a\in \Aut(H)$ with $\gcd(|a|,|H|)=1$ and $z$ is an involution in $H{:}\l a\r$ such that $\l a,z\r=H{:}\l a\r$.
\end{problem}
This problem is nontrivial, and a useful characterization of balanced Cayley maps and their embeddings is given in
and in \cite{Jozef92}. 
Here we give some examples of $\cayM(G,H,a,z)$.

\begin{example}\label{ex:PSL(2.)}
{\rm
Let $T$ be a simple group whose order is not divisible by $n$. 
Assume that $T$ is generated by three involutions $u,v,w$.
Let $H=T_1\times\dots\times T_n=T^n$, where $T_i\cong T$, and let
\[G=H{:}\l a\r=(T_1\times\dots\times T_n){:}\l a\r,\]
where $a$ has order $n$ and acts as follows:
\[a:\ (t_1,t_2,\cdots,t_n)\mapsto(t_2,\cdots,t_n,t_1).\]
Let $z=(u,v,w,1,\cdots,1)$, an involution of $H=T_1\times\dots\times T_n$.
Then
\[\begin{array}{rll}
z&=&(u,v,w,1,\cdots,1),\\
z^{a}&=&(v,w,1,\cdots,1,u),\\
 z^{a^2}&=&(w,1,\cdots,1,u,v),\\
&\dots\\
z^{a^{n-1}}&=&(1,u,v,w,\cdots,1).\\
\end{array}\]
Since $\l u,v,w\r=T$ and $a$ fixes $\l z,z^a,\cdots,z^{a^{n-1}}\r$, 
it  follows that either $\l z,z^{a},\cdots,z^{a^{n-1}}\r= H$ or $\l z,z^{a},\cdots,z^{a^{n-1}}\r$ is a diagonal subgroup of $H$.
Since $z\in \l z,z^{a},\cdots,z^{a^{n-1}}\r$ and $z$ is not a diagonal element, we have $\l z,z^{a},\cdots,z^{a^{n-1}}\r= H$.
Thus $(a,z)$ is a rotary pair of $G$.
Now $H$ is a Hall subgroup of $G$, and the rotary pair $(a,z)$ determines a Cayley graph $\Cos(G,\l a\r,\l z\r)$, and defines $\cayM(G,H,a,z)$.
\qed
}
\end{example}
We remark in Example~\ref{ex:PSL(2.)} that such an  $n$ always exists since we can choose $n$ to be a prime such that $n>|T|$.

\begin{example}\label{ex:ns}
    {\rm
    Let $T_1,\dots, T_r$ be non-isomorphic simple groups and let  $n_1,\dots,n_r$ satisfy  $\gcd(n_i,n_j)=1$ and $\gcd(n_i,\prod_{j=1}^r|T_j|)=1$. Assume that $T_i$ is generated by three involutions $u_i,v_i,w_i$.

    Let $G_i=T_i^{n_i}{:}\l a_i\r\cong T_i^{n_i}{:}\Z_{n_i}$, where $a_i$ acts as follows: \[a_i:\ (t_1,t_2,\cdots,t_{n_i})\mapsto(t_2,\cdots,t_{n_i},t_1).\]
    Let $z_i=(u_i,v_i,w_i,1,\dots,1)$, an involution of $H_i=T_i^{n_i}$. By Example~\ref{ex:PSL(2.)}, $(a_i,z_i)$ is a rotary pair of $G_i$.

    Let $a=(a_1,\dots,a_r)$, $z=(z_1,\dots,z_r)$ and let $G=G_1\times\cdots\times G_r$. Since $\l a,z\r$ projects onto $G_i$, $T_i\ncong T_j$ and $\gcd(n_i,n_j)=1$ for $1\leq i\neq j\leq r$, we conclude that $(a,z)$ is a rotary pair of $G$.
    Now $H=H_1\times\cdots\times H_r$ is a Hall subgroup of $G$, and the rotary pair $(a,z)$ determines a Cayley graph $\Cos(G,\l a\r,\l  z\r)$, and defines $\cayM(G,H,a,z)$.\qed}
\end{example}

\begin{example}\label{ex:ab}
    {\rm Let $G=H{:}\l a\r$, where $H\cong\Z_p^n$ is a vector space over $\mathbb{F}_p$, with $p>2$,  and $\l a\r$ is a subgroup of a Singer cycle of $\GL(n,p)$ such that $|a|=s$ is even and is a primitive divisor of $p^n-1$. Then $a$ acts irreducibly on $H$, and hence, for every non-zero $v\in H$, $\l v^{\l a\r}\r=H$.
    Since $a^{s/2}$ is an involution, its $(-1)$-eigenspace is nontrivial. Choose $v$ to be  an eigenvector corresponding to the eigenvalue $-1$ of $a^{s/2}$, and 
   define $$z=va^{s/2}.$$ Then $z^2=v+v^{a^{s/2}}=v-v=0$, and 
    thus $|z|=2$.
    Moreover, we have  $\l a,z\r=\l a,v\r=G$.
    Hence $(a,z)$ is a rotary pair of $G$ and the map is defined as in Example~\ref{ex:PSL(2.)}.}\qed
\end{example}

\subsection{Non-balanced case}
In this case, we study maps that are not balanced.
We first state a useful lemma.

\begin{lemma}\label{lem:module}
Let $(a,z)$ be a rotary pair of $G$ and let $G^*=V\rtimes_\phi G\cong\ZZ_p^n\rtimes_\phi G$, where $V$ is an irreducible $\mathbb{F}_p G$-module, $\gcd(p,|G|)=1$, and 
$\phi:G\rightarrow \GL(V)$ is the corresponding irreducible representation.
Assume that $\phi(a)$  fixes no non-zero vectors of $V$, and that $\phi(z)\neq 1$. 
Then $$(a,vz)\text{ is a rotary pair of }G^*,$$ where $v\in V$ is an eigenvector of $\phi(z)$ corresponding to the eigenvalue $-1$.
\end{lemma}
\begin{proof}
Since $(a,z)$ is a rotary pair of $G$, we have $|G|$ is even.
As $\gcd(p,|G|)=1$, we conclude that $p\neq 2$.
Thus $\phi(z)$ has a nontrivial $(-1)$-eigenspace because $|\phi(z)|=2$.
So there exists $0\neq v\in V$ such that $v^z=-v$.
It is straightforward to check that $|vz|=2$. Set $K=\l a,vz\r$, and we only need to show $G^*=K$. 
Since $G=\l a,z\r$, we have  $$K V/V=\l aV,vzV\r=\l aV, zV\r =G^*/V,$$  and hence $K V=G^*$.
  Since $K\cap V\lhd K$, we have  $K\cap V\lhd KV=G^*$.
 It means that $K\cap V$ is a $G$-invariant subspace of $V$. As $V$ is an irreducible $\mathbb{F}_pG$-module, 
 $${\rm either}\, K\cap V=1, \,{\rm  or} \,K\cap V=V.$$
 If $K\cap V=V$, then $V\le K$, and hence $G^*=KV=K$.
  Suppose, to the contrary, that $K\cap V=1$. 
  Since $\gcd(p,|G|)=1$, by the Schur-Zassenhaus Theorem there exists some $u\in V$ such that $K^u=G$. 
  Therefore, $a^u\in G$. 
  However, $a^{-1}a^u=a^{-1}u^{-1}au=(u^{-1})^au\in V\cap G$, and since $G\cap V=1$ we have $(u^{-1})^au=1$. It follows that $(u^{-1})^a = u^{-1}$, and thus $u=0$ since $\phi(a)$ fixes no no-zero vectors of $V$. Now $K=G$,  which is impossible since $vz\notin G$.
\end{proof}

Now we are ready to give the constructions. We first give the covers of the maps defined by the rotary pairs as in Theorem~\ref{thm:rot}~${\rm (i)}$.

\begin{construction}\label{con1}
  Take $G,H,a,z$ as in Theorem~\ref{thm:rot}~${\rm (i)}$, i,e, $G=\l a\r{:}\l z\r$. 
  Set $$G^*=V\rtimes_{\phi}G\cong\ZZ_p^n\rtimes_{\phi}G,\ \text{ and } H^*=V\rtimes_{\phi}\l z\r\cong\ZZ_p^n\rtimes_{\phi}\l z\r,$$ where $\gcd(p,|G|)=1$, $\phi$ is an irreducible representation of $G$ on $V$ such that $\phi(z)\neq 1$.
  Take  $v\in V$ to be  an eigenvector  of $\phi(z)$ corresponding to the eigenvalue $-1$.
\end{construction}
\begin{proposition}\label{pro:11.7}
With the notation in Construction~\ref{con1}, the following statements hold:
\begin{enumerate}[\rm(i)]
    \item such  a representation $\phi$ exists;
    \item $(a,vz)$ is a rotary pair of $G^*$;
    \item $\cayM(G^*,H^*,a,vz)$ is a cover of $\cayM(G,H,a,z)$.
\end{enumerate}
\end{proposition}

\begin{proof}
We first prove the existence of $\phi$. 
Let $(\Phi,\mathbb{F}_p^{|G|})$ be the regular representation of $G$ over $\mathbb{F}_p$. 
Since $\gcd(p,|G|)=1$, $\mathbb{F}_p^{|G|}$ is semisimple by Maschke Theorem (see \cite[Theorem 2.11]{book:char}), and thus one can deduce
    $$\Phi = \phi_1^{r_1}\times\cdots\times \phi_n^{r_n}$$
for some integers $r_1,\dots, r_n$, where $\phi_1,\dots,\phi_n$ are irreducible representations of $G$. 
Since $\Phi(z)\neq 1$, there exists some $\phi_i$ such that $\phi_i(z)\neq 1$. 
We take $\phi=\phi_i$, and (i) holds.

Now we claim that $(a,vz)$ is a rotary pair of $G^*$. 
By Lemma~\ref{lem:module}, we only need to check that $\phi(a)$ does not fix any non-zero vector in $V$. 
Since $\phi: G\rightarrow \GL(V)$ is irreducible and $\l a\r\lhd G$, by Clifford theorem (see \cite[Theorem 19.3]{book:char}), the restriction $\phi|_{\l a\r}$ is either irreducible or the sum of two irreducible representations $\psi_1+\psi_2$ of $\l a\r$.
Therefore, it suffices to show that $\psi(a)$ fixes no non-zero vector for every irreducible representation $\psi$ of $\l a\r$ over $\mathbb{F}_p$.
Suppose that $\psi$ is any nontrivial faithful irreducible representation of $\l a\r$ over $\mathbb{F}_p$ with $\gcd(p,|a|)=1$. 
Then $\psi(a)$ has dimension $m$, where $m$ is the minimal integer such that $p^m\equiv 1\pmod {|\psi(a)|}$. 
Moreover, $\langle \psi(a)\rangle$ is a subgroup of a Singer cycle of $\GL(m,p)$, and thus $\psi(a)$ fixes no non-zero vectors of $V$. 
If $\psi$ is unfaithful with kernel $\l a^k\r$, then since $\l a\r/\l a^k\r\cong \l \ov {a}\r$ is cyclic, $\psi(a)$ fixes no non-zero vectors either. So (ii) holds.

By the definition of $G^*$ and $H^*$, we have $G^*=H^*\l a\r$.
Let $\mathcal{N}=\cayM(G^*,H^*,a,vz)$.  One can check that $\cayM(G,H,a,z)=\mathcal{N}_V$. Hence (iii) holds.
\end{proof}


Now we give the covers of the maps defined by the rotary pairs as in Theorem~\ref{thm:rot}~{\rm (ii)}.
\begin{construction}\label{con2}
  {\rm
  Take $G,H,a,z$ as in Theorem~\ref{thm:rot} $\rm(ii)$. For each $i$, let $G_i=T_i{:}\l c_i\r$.
Then  $G_i$ is a $2$-transitive group, and $G\le (\l a_0\r{:}\l c_0\r)\times G_1\times \cdots \times G_n$.
Let $V_i$ be 
a nontrivial irreducible constituent of the corresponding $2$-transitive permutation $\mathbb{F}_pG_i$-module, and let 
$$\pi_i: G_i\rightarrow \GL(V_i)$$ denote the associated representation; here $p\nmid |G|$. 
  For a given $i$, define a representation
  $$\pi: G\rightarrow \GL(V_i), \,\,(g_0,g_1,\cdots,g_n)\mapsto  \pi_i(g_i).$$
  Set
  $$G^*=V_i\rtimes_{\pi} G,\ \text{ and } H^*=V_i\rtimes_{\pi} H.$$
 Take  $v\in V_i$ to be an eigenvector of $\pi(z)$ corresponding to the eigenvalue $-1$.}
\end{construction}
\begin{proposition}\label{pro:11.9}
 With the notation in Construction~\ref{con2}, the following statements hold:
\begin{enumerate}[\rm(i)]
    \item $(a,vz)$ is a rotary pair of $G^*$;
    \item $\cayM(G^*,H^*,a,vz)$ is a cover of $\cayM(G,H,a,z)$.
\end{enumerate}
\end{proposition}
\begin{proof}
According to Table \ref{Tab:AS-candidates},  there exists a Hall subgroup $H_i$ of $G_i$ such that $G_i=H_i\l a_i\r$ and $H_i$ is core-free in $G_i$. 
Suppose $G_i$ is a $2$-transitive group on $\Omega:=[G_i:H_i]$ with  $|\Omega|=n_i$, and $P=W_i\rtimes_{\psi_i}G_i\cong\ZZ_p^{n_i}\rtimes_{\psi_i}G_i$, where $p$ is a prime with $p\nmid |G|$ and $\psi_i$ is the corresponding  $2$-transitive permutation representation of $G_i$. 
Assume $\{v_1,\dots,v_{n_i}\}$ is a basis of $W_i$ on which $G_i$ acts by permuting the coordinates.  
By \cite[Theorem 11.6]{book:char}, we have $\psi_i = \mathbf{1}\oplus \pi_i$, where $\mathbf{1}$ is the trivial representation of $G_i$ on $U_i=\langle v_1+\cdots+v_{n_i}\rangle$ and $\pi_i$ is the irreducible representation of $G_i$ on $V_i=\langle v_1-v_2, v_2-v_3,\ldots,v_{n_i-1}-v_{n_i}\rangle$.


Now we prove that $(a,vz)$ is a rotary pair of $G^*$. 
By Lemma~\ref{lem:module}, it is sufficient to show that $\pi(z)\neq 1$ and $\pi(a)$ fixes no non-zero  vectors of $V_i$.
Define  $$\psi:G\rightarrow \GL(W_i),\,\, (g_0,g_1,\cdots, g_n)\mapsto \psi_i(g_i).$$
By the definition of $\psi$ and $\pi$, we have $\psi=\mathbf{1}\oplus\pi$ since $\psi_i=\mathbf{1}\oplus \pi_i$.

Since $\psi(z)=\psi_i(z_i)$ is nontrivial and $\psi(z)=\mathbf{1}(z)\oplus\pi(z)=\pi(z)$, $\pi(z)$ is nontrivial.
Suppose that $w=k_1(v_1-v_2)+\cdots+k_{n_i-1}(v_{n_i-1}-v_{n_i})\in V_i$ satisfies $w^{\pi(a)}=w$. 
Then 
\begin{align*}
    w^{\pi(a)}&=k_1(v_2-v_3)+\cdots+k_{n_i-2}(v_{n_i-1}-v_{n_i})+k_{n_i-1}(v_{n_i}-v_1)\\
    &=-k_{n_i-1}(v_1-v_2)+(k_1-k_{n_i-1})(v_2-v_3)+\cdots+(k_{n_i-2}-k_{n_i-1})(v_{n_i-1}-v_{n_i})\\
    &=k_1(v_1-v_2)+\cdots+k_{n_i-1}(v_{n_i-1}-v_{n_i}).
\end{align*}
Thus we conclude that $k_1=-k_{n_i-1}$ and $k_{j}=k_{j-1}-k_{n_i-1}$ for $2\leq j\leq n_i-1$.
It follows that $n_ik_{n_i-1}=0$. Since $n_i$ is a divisor of $|G_i|$ and the order of $G_i$ divides that of $|G|$, we have $n_i$ is a divisor of $|G|$. Noting that $\gcd(p,|G|)=1$, we conclude that $p\nmid n_i$ and thus $k_{n_i-1}=0$. 
Hence each $k_j=0$ and $w=0$. 
Therefore, $\pi(a)$ fixes no non-zero element in $V_i$, and (i) holds.

Let $\mathcal{N}=\cayM(G^*,H^*,a,vz)$, and  one can check that $\cayM(G,H,a,z)=\mathcal{N}_V$. Hence (ii) holds.
This completes the proof.
\end{proof}

\begin{proof}[{\bf Proof of Theorem~\ref{regular}:}]
Let $\calM=\cayM(G,H,a,z)$, where $H$ is a Hall subgroup of $G$.
We prove this theorem by considering  two cases depending on whether $H$ is core-free in $G$.

\noindent\textbf{Case 1: Assume that $H$ is core-free in $G$.} 

It suffices to construct infinitely many elementary abelian groups $V$ such that, for $G^*=V{:}G$, there exists a rotary pair $(a,\rho)$. In this case, $\CayM(G^*, V{:}H, a,\rho)$ is a cover of $\cal M$.
Let
$$V=V_1\times\cdots\times V_r\cong \Z_{p_1}^{n_1}\times\cdots\times \Z_{p_r}^{n_r},$$ where $\gcd(p_i,p_j|G|)=1$ for $1\leq i\neq j\leq r$. 
For each $i$, let $\phi_i$ be an irreducible representation of $G$ on $V_i$ such that $\phi_i(a)$ fixes no elements of $V_i$ and $\phi_i(z)\neq 1$. 
Such $\phi_i$'s always exist by Propositions~\ref{pro:11.7} and \ref{pro:11.9}.
Since $z$ is an involution and $\phi_i(z)\neq 1$, we may choose a non-zero vector $v_i\in V_i$ such that $v_i^{\phi_i(z)}=-v_i$.
Set $$v=(v_1,\cdots,v_r)\in V, \, \phi=\phi_1\times\cdots\times \phi_r, G^*=V\rtimes_{\phi}G.$$
We claim that $(a,vz)$ is a rotary pair of $G^*$.  We prove this by induction on $r$.
If $r=1$, then the assertion follows directly from  Lemma~\ref{lem:module}. 
Now assume  that the assertion holds when $V$ has $m$ direct factors, and suppose that $V=V_1\times \cdots \times  V_m\times V_{m+1}$.
Put $W=V_1\times \cdots\times V_m$.
Since $V_{m+1}\lhd G^*$, we have $G^*/V_{m+1}\cong W\rtimes G$. By the induction hypothesis,
$$(aV_{m+1},vzV_{m+1})\text{ is a rotary pair of }G^*/V_{m+1}.$$ 
Observe that  $G^*=V_{m+1}\rtimes _{\phi_{m+1}}(G^*/V_{m+1})$. Applying Lemma~\ref{lem:module} to the quotient $G^*/V_{m+1}$ , we conclude
that  $(a,vz)$ is a rotary pair of $G^*$. This completes  the induction.



\noindent\textbf{Case 2: Suppose  $H$ is not core-free in $G$.} 

 Set $L=\core_G(H)$. 
Then $H/L$ is core-free in $G/L$.
Applying Case 1 to $G/L$, choose $V=V_1\times \cdots \times V_r$ and a representation $\bar{\phi}=\phi_1\times \cdots\times \phi_r: G/L\rightarrow \GL(V)$
satisfying the conditions obtained in Case 1.
Let $$\phi: G\rightarrow \GL(V), \,\, g\mapsto \bar{\phi}(gL),$$ which is  the lift  of  $\bar{\phi}$.
Then $L\le \ker(\phi)$. 
Set $$G^*=V\rtimes_\phi G.$$
Choose $v=(v_1,\cdots, v_r)\in V$, where each $v_i$ is a $(-1)$-eigenvector  of $\phi_i(z)$.
Similar to Case 1, by induction on the number of direct factors of $V$, and by using Lemma~\ref{lem:module} repeatedly, we conclude that $(a,vz)$ is a rotary pair of $G^*$.
Therefore, $\CayM(G^*,V{:}H,a,vz)$ is a cover of $\CayM(G,H,a,z)$.


In conclusion, $\cayM(G^*,V{:}H,a,vz)$ is a cover of $\cayM(G,H,a,z)$, where $v\in V$ and $V=\Z_{p_1}^{n_1}\times\cdots\times\Z_{p_r}^{n_r}$. 
When $G$ is given, the choices of $(p_1,\cdots,p_r)$ are infinite since we can take any $r$ distinct primes that are larger than $|G|$.
Therefore, there are infinitely many choices for $V$ and infinitely many covers. 
\end{proof}

\end{document}